\documentclass[11pt]{article}
\usepackage[margin=1in]{geometry}

\usepackage{tikz}
\usetikzlibrary{arrows.meta}
\usepackage{listing}
\usepackage{algorithm}
\usepackage{graphicx}
\usepackage{subcaption}
\usepackage{multicol}
\usepackage{wrapfig}
\usepackage{makecell}
\usepackage{amsfonts}
\usepackage{mathtools}
\usepackage{amssymb}
\usepackage{amsmath}
\usepackage{etoolbox}

\newtheorem{thm}{Theorem}
\newtheorem{lemma}{Lemma}
\newtheorem{prop}{Proposition}
\newtheorem{cor}{Corollary}

\newtheorem{definition}{Definition}
\newtheorem{myexample}{Example}
\newtheorem{remark}{Remark}

\usepackage{algpseudocode}
\usepackage[colorlinks=true,breaklinks=true,bookmarks=true,urlcolor=blue,
citecolor=blue,linkcolor=blue,bookmarksopen=false,draft=false]{hyperref}
\newenvironment{proof}{\emph{Proof.}}{\(\blacksquare\)}

\newenvironment{customproof}[1]{\emph{Proof of #1.}}{\(\blacksquare\)}

\AtBeginEnvironment{pmatrix}{\setlength{\arraycolsep}{1pt}}
\newcommand{\coloneqq}{\mathrel : =}
\newcommand{\weakcvg}{\rightharpoonup}
\newcommand{\Tr}{\operatorname{Tr}}
\newcommand{\LintRn}{{\bf L^n}}
\newcommand{\LintR}[1]{{\bf L^{#1}}}

\newcommand{\SintRnn}{{\bf S^m}}
\newcommand{\Momn}{{\bf M}}

\newcommand{\bachir}[1]{{#1}}

\title{\bf Time-Varying Semidefinite Programs}
\date{}

\author{
 Amir Ali Ahmadi and  Bachir El Khadir \thanks{The authors are with  the department of Operations Research and Financial Engineering at Princeton University. \newline
   This work was partially supported by the DARPA Young Faculty Award, the MURI award of the AFOSR, the CAREER Award of the NSF, the Google Faculty Award, the Innovation Award of the School of Engineering and Applied Sciences at Princeton University, and the Sloan Fellowship.}\\
 \{\texttt{a\_a\_a}, \texttt{bkhadir}\}\texttt{@princeton.edu}
}

\begin{document}

\maketitle

\begin{abstract}
  We study time-varying semidefinite programs (TV-SDPs), which are semidefinite programs whose data (and solutions) are functions of time. Our focus is on the setting where the data varies polynomially with time. We show that under a strict feasibility assumption, restricting the solutions to also be polynomial functions of time does not change the optimal value of the TV-SDP. Moreover, by using a Positivstellensatz on univariate polynomial matrices, we show that the best polynomial solution of a given degree to a TV-SDP can be found by solving a semidefinite program of tractable size. We also provide a sequence of dual problems which can be cast as SDPs and that give upper bounds on the optimal value of a TV-SDP (in maximization form). We prove that under a boundedness assumption, this sequence of upper bounds converges to the optimal value of the TV-SDP. Under the same assumption, we also show that the optimal value of the TV-SDP is attained. We demonstrate the efficacy of our algorithms on a maximum-flow problem with time-varying edge capacities, a wireless coverage problem with time-varying coverage requirements, and on bi-objective semidefinite optimization where the goal is to approximate the Pareto curve in one shot.
\end{abstract}

{\bf Index Terms.}
 Semidefinite programming,
 time-varying convex optimizaiton,  
 univariate polynomial matrices,
 Positivstellensatz\" e,
 continuous linear programs, 
 bi-objective optimization.

\section{Introduction}
\label{sec:orgheadline1}

We study semidefinite programs (SDPs) whose feasible set and objective function depend on time. More specifically, a \emph{time-varying semidefinite program} (TV-SDP) is an optimization problem of the form
\newcommand \talmostsurely {\forall t \in [0, 1]  \; \text{a.e.}}
\newcommand \opt {\text{opt}}
\begin{equation}
\label{eqn:time_varying_sdp_l2}
\begin{array}{ll@{}ll}
\underset{x \in \LintRn}{\sup} & \int_0^1 \langle c(t), x(t) \rangle dt & \\
\text{subject to}&  Fx(t)  \succeq 0 &\; \talmostsurely.\\
\end{array}
\end{equation}

Here, the operator $F: \LintRn \rightarrow \SintRnn$ is defined as 
\begin{equation}
\label{eqn:F}
Fx(t) \coloneqq A_0(t) + \sum_{i=1}^n x_i(t) A_i(t)  +  \sum_{i=1}^{n} \int_0^t x_{i}(s) D_i(t, s) ds,
\end{equation}

where 
$$\LintRn \coloneqq \{ x:[0, 1] \rightarrow \mathbb R^n \; | \; x \text{ measurable and } \sup_{t \in [0, 1], i = 1,\ldots,n} | x_i(t) | < \infty\},$$
and
$$\SintRnn \coloneqq \{ X:[0, 1] \rightarrow \mathbb R^{m\times m} \; | \; X(t) \text{ is symmetric }\; \forall t \in [0, 1] \text{ and }  \sup_{t \in [0, 1], i,j =1,\ldots, m} | X_{ij}(t) | < \infty\}.$$

The data to the problem \bachir{consist} of $c \in \LintRn$, $A_i \in \SintRnn$ for $i \in \{0, \ldots, n\}$, and $D_i$ for $i \in \{1, \ldots, n\}$, which satisfies the requirement that $D_i(t, .)$ be a measurable function in $\SintRnn$ for all $t \in [0, 1]$ and that  $\sup_{t,s \in [0, 1]} \|D_i(t, s)\| < \infty$, where $\|.\|$ is any matrix norm. For a symmetric matrix $M$, we write $M \succeq 0$ to denote that $M$ is positive semidefinite, i.e., has nonnegative eigenvalues. The abbreviation \emph{a.e.} indicates that the matrix inequality in (\ref{eqn:time_varying_sdp_l2}) should hold ``almost everywhere''; i.e., for every \(t \in [0, 1] \setminus N\), where \(N\) is some set of measure zero with respect to the Lebesgue measure.

For an interval $I \subseteq \mathbb R$, we define the set $$\SintRnn^+(I) \coloneqq \{ X \in \SintRnn \; | \; X(t) \succeq 0 \quad \forall t \in I \; \text{a.e.}\}.$$ With this notation, a \emph{feasible solution} to the TV-SDP in \eqref{eqn:time_varying_sdp_l2} is a function $x \in \LintRn$ that satisfies the constraint
\begin{equation}
  \label{eqn:constraint_tvsdp}
  Fx \in \SintRnn^+([0, 1]),
\end{equation}
and the \emph{feasible set} of the TV-SDP is the set $\mathcal F \coloneqq \{ x \in \LintRn \; | \; Fx \in \SintRnn^+([0, 1])\}$. The choice of the interval \([0, 1]\) is of course made for convenience. Without loss of generality, we can reduce any bounded interval \([a, b]\), with $a < b$, to the interval $[0, 1]$ by performing the change of variable \(t' = \frac{t-a}{b-a}\).

We equip $\LintRn$ and $\SintRnn$ respectively with the inner products $\langle \cdot, \cdot \rangle_{\LintRn}$ and  $\langle \cdot, \cdot \rangle_{\SintRnn}$ defined as
$$\langle x, y \rangle_\LintRn \coloneqq \int_0^1 \langle x(t), y(t) \rangle \; dt = \sum_{i=1}^n\int_0^1 x_i(t)y_i(t) \; dt,$$
and
$$\langle X, Y \rangle_\SintRnn \coloneqq \int_0^1 \langle X(t), Y(t) \rangle \; dt = \int_0^1 \Tr(X(t)Y(t)) \; dt,$$
where $\Tr(A)$ stands for the trace of a matrix $A$. Using the notation for the first inner product above, the TV-SDP in \eqref{eqn:time_varying_sdp_l2} can be written more compactly as

\begin{equation*}
\label{eqn:time_varying_sdp_compact}
\begin{array}{ll@{}ll}
\underset{x \in \LintRn}{\sup} & \langle c, x \rangle_{\LintRn} & \\
\text{subject to}&  Fx  \in \SintRnn^+([0, 1]).
\end{array}
\end{equation*}

The terms \(\int_0^t x_{i}(s) D_i(t, s) ds\) in \eqref{eqn:F} are called \emph{kernel terms} and broaden the class of problems that can be modelled as a TV-SDP. The special case where the terms \(D_i(t, s)\) are identically zero is already interesting and presents an infinite sequence of SDPs indexed by time $t \in [0, 1]$. While these SDPs are in principle independent of each other, basic strategies such as sampling $t$ and solving a finite number of independent SDPs generally fail to provide a solution to the TV-SDP. This is because candidate functions obtained from simple interpolation schemes can violate feasibility in between sample points. When the terms \(D_i(t, s)\) are not zero, the value that a solution takes at a given time affects the range of values that it can take at other times. When the terms \(D_i(t, s)\) are constant functions of \(t\) and \(s\) for instance, the TV-SDP in \eqref{eqn:time_varying_sdp_l2} is already powerful enough to express linear constraints involving the function $x$ and its derivatives and/or integrals of any order. For example, to impose a constraint on $x'(t)$, one can introduce a new decision variable $y \in \LintRn$ which is related to $x$ via the linear constraint $x(t) - \int_0^t y(s) \; {\rm d}s =0 $.

In this paper, we consider the data $c, A_0, \ldots, A_n, D_1, \ldots, D_n$ of the TV-SDP in \eqref{eqn:time_varying_sdp_l2} to belong to the class of \emph{polynomial} functions. Our interest in this setting stems from two reasons. On the one hand, the set of polynomial functions is dense in the set of continuous functions on \([0, 1]\) and hence powerful enough for modeling purposes. On the other hand, polynomials can be finitely parameterized (in the monomial basis for instance) and are very suitable for algorithmic operations. 

Even when the input data to a TV-SDP is polynomial, there is no reason to expect its optimal solution to be a polynomial or even a continuous function. Nevertheless, we concern ourselves in this paper with finding feasible polynomial solutions to a TV-SDP (which naturally provide lower bounds on its optimal value). Our motivation for making this choice is twofold. First, solutions that are smooth are often preferred in practice. Consider for example the problem of scheduling generation of electric power when daily user consumption varies with time, or that of finding a time-varying controller for a robotic arm that serves some routine task in a production line. In such scenarios, smoothness of the solution is important for avoiding deterioration of the hardware or guaranteeing safety of the workplace. Continuity of the solution is even more essential as physical implementation of a discontinuous solution is not viable. 

Our second motivation for studying polynomial solutions is algorithmic. We will show (cf. Section~\ref{sec:tvsdp_is_sdp}) that optimal polynomial solutions of a given degree to a TV-SDP with polynomial data can be found by solving a (non time-varying) SDP of tractable size.

These observations call for a better understanding of the power of polynomial solutions as their degree increases, or a methodology that can bound their gap to optimality when their degree is capped. These considerations are the subjects of Section~\ref{sec:poly_are_optimal} and Section~\ref{sec:dual-approach} respectively.

As an illustration of a TV-SDP with polynomially time-varying data and a preview of our solution technique, consider problem \eqref{eqn:time_varying_sdp_l2} with $n=2$ and data

$$A_0(t) = \begin{pmatrix}(1-\frac85t)^2&0&0&0\\0&0&0&0\\0&0&0&0\\0&0&0&0\end{pmatrix}, A_1(t) = \begin{pmatrix}-1&0&0&0\\0&0&1&0\\0&1&0&0\\0&0&0&0\end{pmatrix}, A_2(t) = \begin{pmatrix}0&0&0&0\\0&0&0&1\\0&0&0&0\\0&1&0&0\end{pmatrix}, D_1(t,s)=D_2(t,s)=0,$$
\[
c(t) = \begin{pmatrix}9t^2 - 9t + 1\\23t^3 - 34t^2 + 12t\end{pmatrix}.
\]

\begin{figure}[htb]
\centering
\includegraphics[width=0.75\textwidth]{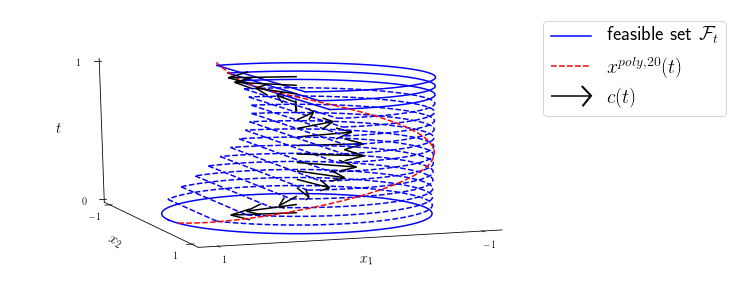}
\caption{\label{img:example_tv_sdp}
An example of a TV-SDP}
\end{figure}

As the kernel terms $D_i(t, s)$ are identically zero here, an optimal solution to this TV-SDP is a function  $x^{\opt} \in {\bf L^2}$ such that for all $t$ in $[0, 1]$ (except possibly on a set of measure zero), $x^{\opt}(t)$ is a maximizer of 
\(\langle c(t), x\rangle\) under the constraints \(A_0(t) + x_1 A_1(t) + x_2A_2(t) \succeq 0\). In Figure \ref{img:example_tv_sdp}, the dotted red line represents the optimal polynomial solution \(x^{\text{poly}, 20}(t)\) of degree \(20\). The feasible set   
$$\mathcal F_t \coloneqq \{x \in \mathbb R^2 \; |\; A_0(t) + x_1 A_1(t) + x_2A_2(t) \succeq 0 \}$$ 
for some sample times \(t\) is delimited by blue lines. The objective function \(c(t)\) is represented by a black arrow, which also moves in time.
The feasible solution \(x^{\text{poly}, 20}(t)\) achieves an objective value of $0.89$. By solving an inexpensive dual problem (with $d = 10$ in problem \eqref{eqn:ud.tvsdp} of  Section \ref{sec:dual-approach}), we can conclude that the optimal value of the TV-SDP cannot be greater than $0.93$. Moreover, 
we can get arbitrarily close to the exact optimal value of the TV-SDP by increasing the degree of the candidate polynomial solutions (cf. Section \ref{sec:poly_are_optimal}) or the level in the hierarchy of our dual problems (cf. Section \ref{sec:weak.and.strong.duality}).

\subsection{Related literature}
Time-varying SDPs contain as special case the time-varying versions of most common classes of convex optimization problems, including linear programs, convex quadratic programs, and second-order cone programs. In the linear programming case, this problem has been studied in the literature under the name of continuous linear programs (CLPs).
In its most general form, a CLP is a problem of the type

\begin{equation}
\label{eqn:clp}
\begin{array}{ll@{}ll}
\underset{x(t)}{\sup} & \int_0^1 \langle c(t), x(t) \rangle dt & \\
\text{subject to}& A(t) x(t) + \int_0^t D(t, s) x(s) ds \le b(t)  & \quad \forall t \in [0, 1]\;\text{a.e.},
\end{array}
\end{equation}
\noindent where \(A(t), D(t, s) \in \mathbb R^{m \times n}, c(t) \in \mathbb R^n\), and \(b(t) \in \mathbb R^m\), for all \(t,s \in [0, 1]\).

This problem was introduced by Bellman \cite{bellman_bottleneck_1953} and has since been studied by several authors who have provided algorithms, structural results, or a duality theory for CLPs; see e.g. \cite{
  lehman_continuous_1954,
  tyndall_duality_1965,
  anderson1987linear,
  levinson_class_1966,
  tyndall_extended_1967,
  grinold_continuous_1969,
  buie_numerical_1973,
  perold_fundamentals_1978,
  anstreicher1984generation,
  anderson_continuous-time_1989,
  luo_new_1998,
  shapiro_duality_2001,
  fleischer_efficient_2005,
  weiss_simplex_2008} and references therein. Several applications, e.g. in manufacturing, transportation, robust optimization, queueing theory, and revenue management, can also be found in these references.

Since CLPs are perceived as a hard problem class in general,
most authors make additional assumptions on how the problem
data varies with time, or, in the case of the so-called
``separated CLPs'' (\bachir{SCLPs}), how the kernel terms
and the non-kernel terms
interact~\bachir{\cite{pullan1996duality,
    pullan1993algorithm, pullan2000convergence,
    luo_new_1998, fleischer_efficient_2005,
    weiss_simplex_2008, buie_numerical_1973,
    perold_fundamentals_1978,
    anderson_continuous-time_1989}}.  \bachir{SLCPs enjoy
  many properties that general CLPs do not. For instance, under mild assumptions, SCLPs with piecewise-polynomial data admit piecewise-polynomial solutions \cite{pullan1995forms}---an attractive feature from an
  algorithmic point view. Unfortunately, the situation for TV-SDPs is not as nice, even without a kernel term. For example, consider problem~\eqref{eqn:time_varying_sdp_l2} with $n=2$ and
  data
        $$A_0(t) = \begin{pmatrix}1 & 0\\0 & 1 \end{pmatrix}, A_1(t) = \begin{pmatrix}1 & 0\\0 & -1 \end{pmatrix}, A_2(t) = \begin{pmatrix}0 & 1\\1 & 0 \end{pmatrix}, D_1(t, s) = D_2(t, s) = 0, \text{ and } c(t) = \begin{pmatrix}t\\1-t\end{pmatrix}.$$
        This TV-SDP has no kernel terms. Furthermore, all its data is constant except for the objective function which varies linearly with time. Its unique optimal solution, however, is easily seen to be $\frac{c(t)}{\|c(t)\|}$, i.e.,
        $$x(t) = \frac{1}{\sqrt{t^2+(1-t)^2}} \begin{pmatrix}t\\1-t\end{pmatrix},$$
        which is not a piecewise-polynomial function.
      }

The closest work in the CLP literature to our work is the paper \cite{bampou_polynomial_2012} by Bampou and Kuhn. The authors of this paper also assume that the data of the their CLP varies polynomially with time and employ semidefinite programming to approximate the optimal solution by polynomial (and piecewise polynomial) functions of time. Our approach here generalizes their nice algorithms and convergence guarantees to the SDP setting. In \cite{bampou_polynomial_2012}, the authors also make use of the rich duality theory of CLPs to get a sequence of upper bounds that converges to the optimal value of \eqref{eqn:clp} under certain conditions. The duality framework that we present in this paper is different in nature and is closer in spirit to the approach in~\cite{lasserre_joint+marginal_2009},~\cite{bampou2011scenario}. As it turns out, it suffices for us to assume boundedness of the primal feasible set to guarantee convergence of our dual bounds to the optimal value of the TV-SDP. 


 The only generalization of continuous linear programs that we are aware of appears in the work of Wang, Zhang, and Yao in \cite{wang_separated_2009}, which makes a number of important contributions to separated continuous conic programs. The assumptions in \cite{wang_separated_2009} are however stronger than the ones we make here. In particular, there are separation assumptions on the kernel and non-kernel terms in \cite{wang_separated_2009} and the data to the problem is assumed to vary only linearly with time. 
Another work related to this paper is the work by Lasserre in \cite{lasserre_joint+marginal_2009}, which studies a parametric polynomial optimization problem of the form
\begin{equation}
\label{eq:lasserre.problem}
\begin{array}{ll@{}ll} 
\underset{x(y) \in \mathbb R^n}{\sup} & \int_{y \in K} f(x(y), y) {\rm d}\phi(y) & \\
\text{subject to}& h_{j}(x(y), y) \ge 0 & \forall j\in\{ 1,\ldots,r\}, \forall y \in K \; \phi\text{-a.e.},\\
\end{array}
\end{equation}
where $\phi$ is a probability distribution on some compact basic semialgebraic set $K \subseteq \mathbb R^s$, and $h_{j}(x, y)$ are \bachir{polynomial} functions of $x$ and $y$. An inequality involving $y$ is valid $\phi\text{-a.e.}$ if it is valid for all $y$ in $K$ except on some set $K'$ with $\phi(K') = 0$.
When the kernel terms in (\ref{eqn:F}) are zero, problem \eqref{eqn:time_varying_sdp_l2} can in theory be put in the form of \eqref{eq:lasserre.problem} by setting $s=1$ and replacing the semidefinite constraint 
with nonnegativity of all $2^m-1$ polynomials that form the principal minors of $Fx(t)$. 
Our duality framework in Section \ref{sec:dual-approach} is inspired by the approach in \cite{lasserre_joint+marginal_2009}. However, as we are dealing with a much more structured problem, we are also able to find the best polynomial solution of a given degree to \eqref{eqn:time_varying_sdp_l2} with an SDP of tractable size, as well as prove asymptotic optimality of polynomial solutions even in presence of the kernel terms.

Finally, we remark that at a broader level, the idea of using semidefinite programming to find polynomial solutions (or ``policies'') to dynamic or uncertain optimization problems has been applied before to questions in multi-stage robust and stochastic optimization; see e.g.~\cite{bertsimas2011hierarchy} and~\cite{bampou2011scenario}.




\renewcommand\labelitemi{{\boldmath$\cdot$}}

\subsection{Organization and contributions of the paper} 
This paper is organized as follows.
In Section \ref{sec:optim-value-is-attained}, we prove that under a boundedness assumption, the optimal value of the TV-SDP in \eqref{eqn:time_varying_sdp_l2} is attained (Theorem \ref{thm:optim-value-attained}). This proof is obtained by combining two theorems that are used also in other sections of the paper. The first (Theorem \ref{thm:bounded.measure.compact}) shows that a sequence of linear functionals that satisfies a certain boundedness property on nonnegative polynomials has a weakly convergent subsequence. The second (Theorem \ref{thm:weak_lim_satisfies_ineq}) shows that when a weakly convergent sequence of functions in $\LintRn$ satisfies linear inequalities of the type in \eqref{eqn:constraint_tvsdp}, then so does its weak limit.

In Section \ref{sec:primal-approach}, we prove that under a strict feasibility assumption, polynomial solutions are arbitrarily close to being optimal to the TV-SDP in \eqref{eqn:time_varying_sdp_l2} (Theorem \ref{thm:poly.optimal}). We also show that this assumption cannot be removed in general (Example \ref{ex:tv.sdp.with.no.poly.sol}). Furthermore, we show how sum of squares techniques combined with certain matrix Positivstellensatz\" e enable the search for the best polynomial solution of a given degree to be cast as an SDP of polynomial size (Theorem \ref{thm:tvsdp_as_sdp}).

In Section \ref{sec:dual-approach}, we develop a hierarchy of dual problems (or relaxations) that give a sequence of improving upper bounds on the optimal value of the TV-SDP in \eqref{eqn:time_varying_sdp_l2}.
We show that under a boundedness assumption, these upper bounds converge to the optimal value of the TV-SDP (Theorem \ref{thm:strong.duality}). We also show that our dual problems can be cast as SDPs (Theorem \ref{thm:dual.is.sdp}). For a given TV-SDP, the dimensions of the matrices that feature in both our primal and dual SDP hierarchies grow only linearly with the order of the hierarchy. 

In Section \ref{sec:applications}, we present applications of time-varying semidefinite programs to a maximum-flow problem with time-varying edge capacities, a wireless coverage problem with time-varying coverage requirements, and to bi-objective semidefinite optimization where the goal is to approximate the Pareto curve in one shot. Finally, we end with some future research \bachir{directions} in Section~\ref{sec:conclusion}.


\newcommand \measureset {{\bf M}^n}

\subsection{Notation}
\label{sec:notation}
We denote
\begin{itemize}
\item the \((i,j)^{th}\) entry of a matrix \(A\) by $A_{ij}$,
\item the trace of a matrix $A$ by $\Tr(A)$,
\item the vector of all ones by $\vec 1$,
\item the identity matrix by $I$,
\item the diagonal matrix with the vector $x\in\mathbb{R}^n$ on its diagonal by $diag(x)$,
\item the standard inner product in $\mathbb R^n$ by $\langle \cdot, \cdot \rangle$; i.e., for two vectors $x, y \in \mathbb R^n$, $\langle x, y \rangle = \sum_{i=1}^nx_i y_i$,
\item the infinity-norm of a vector by $\|\cdot\|_\infty$; i.e., for a vector $x \in \mathbb R^n$, $\|x\|_\infty = \max_{i=1,\ldots,n} |x_i|$,
\item the set of \(n \times n\) (constant) symmetric matrices by \(\mathcal S^n\) and its subset of positive semidefinite matrices by \({\mathcal S^n}^+\),
\item the degree of a polynomial $p$ by $\deg(p)$ (when $p$ is a vector of polynomials, $deg(p)$ denotes the maximum degree of its entries),
\item the set of \(n \times m\) matrices  whose components are polynomials in the variable \(t\) with real coefficients by \(\mathbb R ^{n \times m}[t]\). For \(d \in \mathbb N\), \(\mathbb R ^{n \times m}_d[t]\) denotes the subset of $\mathbb R ^{n \times m}[t]$ consisting of matrices whose entries are polynomials of degree  at most \(d\). When $m = 1$, we simply use the notation $\mathbb R ^{n}[t]$ and $\mathbb R ^{n}_d[t]$, and when $n = 1$ as well, we simplify the notation to $\mathbb R[t]$ and $\mathbb R_d[t]$.
\item We denote the set of linear functionals on $\mathbb R^{n}[t]$ by $\measureset$. 
\item For $\mu \in \measureset$, we denote by $\mu_i: \mathbb R[t] \rightarrow \mathbb R$ the unique linear functional that satisfies $$\mu(g) = \sum_{i=1}^n \mu_i(g_i)  \quad \forall g \in \mathbb R^{n}[t].$$
\item For a function $f \in \LintRn$, we denote by $l_f$ the element of $\measureset$ defined by
$$l_f(g) \coloneqq \langle f, g \rangle_{\LintRn} = \int_0^1 \langle f(t), g(t) \rangle {\rm d}t  \quad \forall g \in \mathbb R^{n}[t].$$
\end{itemize}

\section{The Optimal Value of a Bounded TV-SDP is Attained}
\label{sec:optim-value-is-attained}
\label{sec:analysis.background}
In this section, we study the following question: If the optimal value $opt$ of (\ref{eqn:time_varying_sdp_l2}) is finite (i.e., the problem is feasible and bounded above), does there exist a function $x^*\in \LintRn$ such that $\langle c, x^* \rangle_\LintRn~=~opt$? Many of the arguments given here will be used again in Section \ref{sec:dual-approach} on duality theory.

The question of attainment of the optimal value (i.e., existence of solutions) is a very basic one and has been studied in the continuous linear programming literature already; see e.g. \cite{grinold_continuous_1969}. 
In the TV-SDP case, note that even for standard SDPs that do not depend on time, the optimal value is not always attained unless the feasible set is bounded. We prove in this section that under the following boundedness assumption
 \begin{equation}
 ``\text{$\exists \gamma > 0$ such that for all feasible solutions $x$ to \eqref{eqn:time_varying_sdp_l2},  $\|x(t)\|_{\infty} \le \gamma$ for all $t \in [0, 1]$ a.e.''}, 
   \label{eqn:assump.boundedness}
 \end{equation}
the optimal value of the TV-SDP in \eqref{eqn:time_varying_sdp_l2} is always attained. 
This is not an immediate fact as the search space $\LintRn$ is infinite dimensional. 
  The idea is to prove that a sequence of feasible solutions to a TV-SDP  whose objective value approaches the optimal value must have a converging subsequence and that the limit of the subsequence must also be feasible. It turns out that the right notion of convergence in this context is \emph{weak convergence}. We begin by stating the definition, and then prove that the weak limit of a sequence of feasible solutions is again feasible.

\begin{definition}
[Weak Convergence]

A sequence of linear functionals  $\{\mu^i\}$ in $\measureset$ converges weakly to a linear functional \(\mu^{\infty} \in \measureset\) (we write \(\mu^i \weakcvg \mu^\infty\)) if for all $p \in \mathbb R^n[t],$
$$\mu^i(p)  \rightarrow \mu^\infty(p) \text{ as $i \rightarrow \infty$}.$$

Similarly, a sequence of functions \(\{f^i\}\) in $\LintRn$ converges weakly to a function \(f^{\infty} \in \LintRn\)  (we write \(f^i \weakcvg f^{\infty}\)) if \(l_{f^i} \weakcvg l_{f^{\infty}}\) as $i \rightarrow \infty$.
\label{def:weak_convergence}
\end{definition}

The next theorem shows a compactness result for the set $\measureset$. 
\begin{thm}
  \label{thm:bounded.measure.compact}
  Let $\{\mu^d\}$ be a sequence of linear functionals in $\measureset$. If the following implication holds for every $d \in \mathbb N$ and every polynomial $q \in \mathbb R_d^n[t]$:
  $$q_i(t) \ge 0 \quad \forall t \in [0, 1],\; \forall i\in\{1\ldots,n\} \implies |\mu^d(q)| \le \sum_{i=1}^n \int_0^1 q_i(t) {\rm d}t,$$
then there exists a function $f \in \LintRn$ and a subsequence of $\{\mu^d\}$ that converges weakly to $l_f$.
\end{thm}

In the proof of this theorem, we will invoke the following lemma, which is obtained by a direct application of a result of Lasserre \cite[Theorem 3.12a]{lasserre_moments_2010}.\footnote{To get the statement of the lemma, apply~\cite[Theorem 3.12a]{lasserre_moments_2010} with $n=1, \mathbb K = [0, 1], m=2, g_1 = t, g_2 = 1-t, L_y = \lambda, L_z = l_v,$ where $v\in \LintR1$ is the constant function equal to one, and observe that for any $p\in\mathbb{R}[t],$ the polynomials $p^2g_1, p^2g_2, p^2g_1g_2$ are nonnegative on the interval $[0, 1]$.}
\begin{lemma}[See Theorem 3.12a in \cite{lasserre_moments_2010}]
  \label{lem:lasserre_thm}
  For a linear \bachir{functional} $\lambda \in \Momn^1$, if there exists a scalar $\kappa$ such that the inequalities 
  $$0 \le \lambda(h) \le \kappa \int_0^1 h(t) \;{\rm d}t$$
  hold for every polynomial $h \in \mathbb R[t]$ that is nonnegative on $[0, 1]$, then there exists a function $f \in \LintR{1}$ such that $\lambda(g)=l_f(g), \forall g\in\mathbb R[t].$  
\end{lemma}

\begin{customproof}{Theorem \ref{thm:bounded.measure.compact}}
  The ideas of the proof are inspired by those in \cite[Chap. 7]{walkden_lecture_nodate}.
  Let $\{b_0, b_1, \ldots\}$ be a basis of $\mathbb R^n[t]$ where all entries of the polynomials $b_j$ are of the form $t^k$ for some nonnegative integer $k$. Let $d_j$ denote the maximum degree of the entries of $b_j$.    
  It is clear by assumption that $|\mu^i(b_j)| \le n$ for every $i, j \in \mathbb N$ such that $i \ge d_j$.
Consider the sequence of real numbers  $\{\mu^i(b_0)\}_{i \ge d_0}$. This sequence is bounded in absolute value by $n$. As such, it has a convergent subsequence $\{\mu^{i, (0)}(b_0)\}$.
Next consider $\{\mu^{i, (0)}(b_1)\}_{i \ge d_1}$. Again, this is a sequence of real numbers that is bounded in absolute value by $n$ and so it has a convergent subsequence $\{\mu^{i, (1)}(b_1)\}$.
Iterating this procedure, we obtain, for each integer $r \ge 0$, a subsequence of linear functionals $\{\mu^{i, (r)}\}$ with the property that $\{\mu^{i, (r+1)}\} \subseteq \{\mu^{i, (r)}\}$. Moreover, for all $j, r \in \mathbb N$ with $r\geq d_j$, the sequence of numbers $\{\mu^{i, (r)}(b_j)\}$ converges as $i \rightarrow \infty$. Now consider the diagonal sequence of linear functionals $\{\mu^{i, (i)}\}$. For every $j$, $\{\mu^{i,{(i)}}(b_j)\}$ converges as $i \rightarrow \infty$ as the sequence of linear functionals $\{\mu^{i, (i)}\}_{i \ge d_j}$ is a subsequence of $\{\mu^{i, (d_j)}\}$.
Since the functions $\{b_i\}$ span $\mathbb R^n[t]$ and the elements of the sequence $\{\mu^{i,{(i)}}\}$ are linear functionals, the sequence $\{\mu^{i, (i)}(g)\}$ converges for all polynomial functions $g \in \mathbb R^n[t]$.
Let $\mu^{\infty}$ be the linear functional defined by
\begin{equation}
  \label{eq:mu_infinity}
  \mu^{\infty}(g) = \lim_{i \rightarrow \infty} \mu^{i, (i)}(g) \quad \forall g \in \mathbb R^n[t].
\end{equation}
We have just proven that the sequence $\{\mu^{i,(i)}\}$ converges weakly to $\mu^\infty$. The claim of the theorem would be established if we show that there exists a function $f\in \LintRn$ such that $\mu^\infty(g)=l_f(g), \forall g\in\mathbb R^n[t].$ 
In order to get this statement from Lemma \ref{lem:lasserre_thm}, for $j \in \{1, \ldots, n\}$, let  $\lambda_j \in \Momn^1$ be defined as
$$\lambda_j(w) \coloneqq \int_0^1 w(t) \; {\rm d}t - \mu^\infty_j(w) \quad \forall w \in \mathbb R[t].$$
Let $h \in \mathbb R[t]$ be a polynomial that is nonnegative on $[0, 1]$. Take $h^{(j)} \in \mathbb R^n[t]$ to be the vector-valued polynomial whose entries are all identically zero except for the $j^{\text{th}}$ one that is equal to $h$.
From \eqref{eq:mu_infinity} we see that $$\mu^{\infty}_j(h) = \mu^\infty(h^{(j)}) = \lim_{i \rightarrow \infty} \mu^{i, (i)}(h^{(j)}) = \lim_{i \rightarrow \infty} \mu^{i, (i)}_j(h).$$
Since for $i$ larger than the degree of $h$,
$$|\mu^{i, (i)}_j(h)| = |\mu^{i, (i)}(h^{(j)})| \le  \sum_{k=1}^n \int_0^1 h^{(j)}_k(t) \;{\rm d}t = \int_0^1 h(t) \; {\rm d}t,$$
we have that that $|\mu^\infty_j(h)| \le \int_0^1 h(t) \; {\rm d}t$, and therefore
  $$|\lambda_j(h)| \le \left|\int_0^1 h(t) \; {\rm d}t\right| + |\mu^\infty_j(h)|  \le 2 \int_0^1 h(t) \; {\rm d}t.$$
 Similarly, it is straightforward to argue that $\lambda_j (h)\geq 0$. Hence, by Lemma \ref{lem:lasserre_thm}, for each $j \in \{1, \ldots, n\}$, there exists a function $\hat f_j \in \LintR{1}$ such that $\lambda_j(w)=l_{\hat f_j}(w), \forall w\in\mathbb R[t].$ Therefore, $$\mu^{\infty}_j(w)=\int_0^1 \big(1-\hat f_j(t)\big)w(t) \; {\rm d}t, \forall w\in\mathbb R[t].$$ The function $f\in \LintRn$ that we were after can hence be taken to be $f\mathrel{\mathop:}=(1-\hat f_1,\ldots, 1-\hat f_n)^T.$
\end{customproof}

The next theorem shows that when all functions in a sequence satisfy linear inequalities of the type in \eqref{eqn:constraint_tvsdp}, their weak limit does the same.
\begin{thm}
Let the operator $F$ be as in (\ref{eqn:F}). If a sequence of functions \(\{f_k\}\) in $\LintRn$ converges weakly to a function \(f_{\infty} \in \LintRn\) and satisfies $Ff_k \in \SintRnn^+([0, 1])$ for all \(k \in \mathbb N\), then $Ff_{\infty} \in \SintRnn^+([0, 1]).$
\label{thm:weak_lim_satisfies_ineq}
\end{thm}

To prove this theorem, we need the following lemma, which also implies that the set $\SintRnn^+([0, 1])$ is self-dual. \bachir{This is a generalization of the corresponding statement for the non time-varying case, which states that the cone \({\mathcal S^n}^+\) is self-dual.}

\begin{lemma}
  \label{lem:s.plus.self.dual}
For any function \(Q \in \SintRnn\),  $Q \in \SintRnn^+([0, 1])$ if and only if $$\langle Q, P \rangle_{\SintRnn} \ge 0 \text{ for all } P \in \SintRnn^+([0, 1]) \cap \mathbb R^{m \times m}[t].$$
\end{lemma}

\begin{proof}
  The only if part is straightforward. For the other direction, fix $Q \in \SintRnn$ and assume that \(\langle Q, P \rangle_{\SintRnn} \ge 0\) for all $P \in \SintRnn^+([0, 1]) \cap \mathbb R^{m \times m}[t]$. For $t\in[0,1]$, let $\lambda(t)$ be the smallest eigenvalue of $Q(t)$ and $u(t)$ be an associated eigenvector of norm one. Denote by $1_{\lambda(t) < 0}$ the univariate function over $t\in[0,1]$ that is equal to 1 when $\lambda(t)<0$ and zero otherwise. Let $P^\infty(t)\mathrel{\mathop:}=1_{\lambda(t) < 0} u(t) u(t)^T.$ We claim that $\langle Q, P^\infty \rangle_{\SintRnn} \ge 0$. This would imply that $$\int_0^1 1_{\lambda(t) < 0}\lambda(t) dt=  \langle Q, P^\infty \rangle_{\SintRnn} \geq 0,$$ which proves that $\lambda(t)$ is nonegative almost everywhere on $[0, 1]$; i.e., the desired result.
  


To prove the claim, observe that since continuous functions are dense in the space of bounded and measurable functions on $[0, 1]$ (see e.g. \cite[Theorem 2.19]{adams_sobolev_2003}), for every positive integer $k$, there exist continuous functions $\phi_k: [0, 1] \rightarrow \mathbb R$ and $u_k: [0, 1] \rightarrow \mathbb R^n$ such that
  $$\int_0^1 (\phi_k(t) - 1_{\lambda(t) < 0})^2 {\rm d}t \leq \frac{1}{k} \  \mbox{and} \int_0^1 \|u_k(t) - u(t)\|_\infty^2 {\rm d}t \leq \frac{1}{k}.$$
  Notice that without loss of generality we can assume that for all $k \in \mathbb N$ and $t\in [0,1]$ we have $\phi_k(t) \ge 0$ as
  $$\left||\phi_k(t)| - 1_{\lambda(t) < 0}\right| \le |\phi_k(t) - 1_{\lambda(t) < 0}|.$$
The Stone-Weierstrass theorem (see e.g. \cite{timan_theory_2014}) can now be utilized to conclude that for every positive integer $k$, there exist polynomial functions $\tilde \phi_k: [0, 1] \rightarrow \mathbb R$, $\tilde u_k: [0, 1] \rightarrow \mathbb R^n$ such that
  $$0 \le \tilde \phi_k(t) - \phi_k(t) \le \frac 1k \text{ and } \|\tilde u_k(t) - u_k(t)\|_\infty^2 \le \frac1k \quad \forall t \in [0, 1].$$
  We can thus assume without loss of generality again that the functions $\phi_k$ and $u_k$ are polynomial functions of the variable $t$.

  Now let $P^k(t) = \phi_k(t) u_k(t)u_k(t)^T$.  Then (i) $P^k\in\SintRnn^+([0, 1]) \cap \mathbb R^{m \times m}[t],$ and (ii) $\|P^\infty - P^k\|_{\SintRnn} \rightarrow 0$ as $k\rightarrow\infty$, where $\|.\|_{\SintRnn}$ here denotes the norm associated to the scalar product $\langle .,.\rangle_{\SintRnn}$. From the Cauchy-Schwarz inequality we have  $$|\langle Q, P^\infty \rangle_{\SintRnn} - \langle Q, P^k \rangle_{\SintRnn}| \le \| Q \|_{\SintRnn} \|  P^\infty-P^k\|_{\SintRnn}.$$ As (i) implies that $\langle Q, P^k \rangle_{\SintRnn} \ge 0$ for all $k$, and (ii) implies that the right hand side of the above inequality goes to zero as $k$ goes to infinity, we conclude that $\langle Q, P^\infty \rangle_{\SintRnn}\geq 0.$    
    %
\end{proof}

\begin{customproof}{Theorem \ref{thm:weak_lim_satisfies_ineq}}
  For an element $y \in \LintRn$, we denote by $\tilde y$ the element of $\LintR{n+1}$ defined by $\tilde y \coloneqq \begin{pmatrix}1\\y\end{pmatrix}$. By applying Fubini's double integration theorem on the region $\{(t, s) \in [0, 1]^2 \; | \; s \le t\}$, it is straightforward to see that
  $$\langle Fy, P \rangle_{\SintRnn} = \langle \tilde y, F^*P\rangle_{\LintR{n+1}} \quad \forall y \in \LintR{n}, \; \forall P \in \SintRnn,$$
  where $F^*$ is the adjoint of the affine operator $F$ (see equation \eqref{eqn:Fstar} in Section \ref{sec:dual-approach} for its explicit expression).
Now fix a function $P \in \SintRnn^+([0, 1]) \cap \mathbb R^{m \times m}[t]$. Using the easy direction of Lemma \ref{lem:s.plus.self.dual} and the fact that $Ff_k \in \SintRnn^+([0, 1])$ for all $k$, we have that  $\langle Ff_k, P \rangle_{\SintRnn} \ge 0$ for all $k$. This implies that $\langle \tilde f_k, F^*P \rangle_{\LintR{n+1}} \ge 0$ for all $k$. By weak convergence, we conclude that $\langle \tilde f_{\infty}, F^*P \rangle_{\LintR{n+1}} \ge 0$, implying in turn that $\langle Ff_{\infty}, P \rangle_{\SintRnn} \ge 0$. \bachir{Since} $P$ was arbitrary in $\SintRnn^+([0, 1])\cap \mathbb R^{m \times m}[t]$, using Lemma \ref{lem:s.plus.self.dual} again, we have
$Ff_{\infty} \in \SintRnn^+([0, 1])$.
\end{customproof}

We are now ready to show that a bounded TV-SDP attains its optimal value. 

\begin{thm}
  \label{thm:optim-value-attained}
If the TV-SDP in \eqref{eqn:time_varying_sdp_l2} is feasible and satisfies the boundedness assumption in \eqref{eqn:assump.boundedness}, then there exists a feasible function \(x^\opt \in \LintRn\) that attains its optimal value.
\end{thm}

\begin{proof}
Let $\opt$ denote the optimal value of \eqref{eqn:time_varying_sdp_l2}, which is finite under the assumptions of the theorem. From \eqref{eqn:assump.boundedness}, there exists a scalar $\gamma >0$ such that any feasible solution $x\in \LintRn$ to the TV-SDP satisfies $\|x(t)\|_{\infty} \le \gamma$ for all $t \in [0, 1]$ a.e.. Hence, for any positive integer $k$, there exists a feasible solution $x^k \in \LintRn$, with $ \|x^k(t)\|_\infty \le \gamma \; \forall t \in [0, 1] \text{ a.e.}$, such that
\begin{equation}
\label{eqn:near_opt}
\langle c, x^k \rangle_{\LintRn} \ge \opt - \frac 1k.
\end{equation}

Let us now consider the sequence of linear functionals \(\big\{\frac{l_{x^k}}{\gamma}\big\}\), which satisfies the conditions of Theorem  \ref{thm:bounded.measure.compact}. Therefore, a subsequence of the functions $\{x^k\}$ converges weakly to a limit $x^\infty \in \LintRn$. It is clear by weak convergence that \(x^{\infty}\) achieves the optimal value to \eqref{eqn:time_varying_sdp_l2}, and Theorem \ref{thm:weak_lim_satisfies_ineq} guarantees that $x^\infty$ is feasible to the TV-SDP. Letting $x^\opt=x^\infty$ gives the desired result.
\end{proof}

\section{The Primal Approach: Polynomial Solutions to a TV-SDP}
\label{sec:primal-approach}
We switch our focus in this section to algorithmic questions. 
We show in Section~\ref{sec:tvsdp_is_sdp} that when the data $c, A_0, \ldots, A_n, D_1, \ldots, D_n$ to our TV-SDP belongs to the class of polynomial functions, then the best polynomial solution of a given degree to the TV-SDP can be found by solving a semidefinite program of tractable size.
This motivates us to study whether one can always find feasible solutions to a TV-SDP that are arbitrarily close to being optimal just by searching over polynomial functions. While this is not always true (see Example \ref{ex:tv.sdp.with.no.poly.sol} below), in Section \ref{sec:poly_are_optimal} we show that it is true under a strict feasibility assumption (see Definition \ref{def:strict_feasibility_sdp}).



\begin{myexample}
  \label{ex:tv.sdp.with.no.poly.sol}
     Consider the TV-SDP in \eqref{eqn:time_varying_sdp_l2} with $n=1$, 
     \[c(t) = 0, A_0(t) = \begin{pmatrix}0&0&0&0\\0&\frac12-t&0&0\\0&0&1&0\\0&0&0&0\end{pmatrix}, A_1(t) = \begin{pmatrix}t-\frac12&0&0&0\\0&t-\frac12&0&0\\0&0&-1&0\\0&0&0&1\end{pmatrix}, \text{ and } D_1 = 0.\]
     The resulting constraints read
     $$\left(t-\frac12\right)x(t) \ge 0, \; \left(t-\frac12\right)(x(t)-1) \ge 0, \; 0 \le x(t) \le 1 \quad \forall t \in [0, 1] \text{ a.e.}.$$
   The unique feasible solution $x^{opt}(t)$ to this TV-SDP, up to a set of measure zero, is
   \[x^{opt}(t) = \left\{\begin{array}{ll}0,&\ \mbox{if}\ t \le \frac12, \\1,&\ \mbox{if}\ t > \frac12. \end{array}\right.\]
   It is clear that $x^{opt}$ is not continuous, let alone polynomial.
\end{myexample}


For the remainder of this paper, for a set $S \in \{\SintRnn, \SintRnn^+([0, 1])\}$ and a nonnegative integer $d$, we define $S_d$ to be the set of \bachir{functions} $x \in S$ whose entries are polynomials of degree $d$, i.e. $$\SintRnn_d = \SintRnn \cap \mathbb R_d^{m \times m}[t], \quad \SintRnn^+([0, 1])_d = \SintRnn^+([0, 1]) \cap \mathbb R_d^{m \times m}[t].$$

\subsection{Polynomials are optimal under a strict feasibility assumption}
\label{sec:poly_are_optimal}


We show in this section that under the following strict feasibility assumption, the optimal value of the TV-SDP in \eqref{eqn:time_varying_sdp_l2} remains the same when the function class $\LintRn$ is replaced with $\mathbb R^n[t]$. 


\begin{definition}\label{def:strict_feasibility_sdp}
  We say that the TV-SDP in \eqref{eqn:time_varying_sdp_l2} is strictly feasible if there exists a function $x^{s} \in \LintRn$ and a positive scalar \(\varepsilon\) such that
  $$Fx^s(t)  \succeq \varepsilon I \quad \forall t \in [0, 1] \; \text{a.e.}.$$
\end{definition}


\begin{thm}
  \label{thm:poly.optimal}
  Consider the TV-SDP in \eqref{eqn:time_varying_sdp_l2} with its optimal value denoted by $\opt$. If the TV-SDP is strictly feasible, then there exists a sequence of feasible polynomial solutions $\{x^k\}$ such that $$\langle c, x^k \rangle_\LintRn \rightarrow \opt \ \mbox{as}\ k\rightarrow \infty.$$
\end{thm}

As we will shortly see in the proof, the strict feasibility assumption enables us to approximate any feasible solution of \eqref{eqn:time_varying_sdp_l2} by a continuous, and later polynomial, solution. We use \emph{mollifying operators} to obtain the continuous approximation. 

\begin{definition}
  [See \cite{adams_sobolev_2003}]
  \label{def:mollifiers}
 The  \emph{mollifying operator} $\mathcal M_v: \LintR1 \rightarrow \LintR1$, indexed by a nonnegative integer $v$, is the linear operator  defined by
  $$(\mathcal M_v f) (t) = \int_0^1 v J(v(t-s)) f(s) {\rm d}s \quad  \forall f \in \LintR1$$
  where $J(t) = c\exp(-\frac{1}{1-t^2})$ when $t \in [-1, 1]$ and $J(t) = 0$ otherwise, and $c$ is so that $\int_{\mathbb R} J(t) \; {\rm d}t = 1$.
\end{definition}
\begin{remark}
\label{rmk:mollifiers}
	To lighten our notation, we write $\mathcal M_v f (t)$ instead of $(\mathcal M_v f) (t)$. We also remark that one can extend the definition of mollifying operators to functions that are not scalar valued by making them act element-wise. For example, the extension to spaces $\LintRn$ and $\SintRnn$ would be defined as follows:
  $$\mathcal M_v f \coloneqq (\mathcal M_v f_i)_i \; \forall f\in \LintRn \text{ and } \mathcal M_v P \coloneqq (\mathcal M_v P_{ij})_{ij} \; \forall P \in \SintRnn.$$
  Any property of mollifying operators that we prove on scalar-valued functions below
  extends in a straightforward manner to functions that are vector or matrix valued.
\end{remark}

\begin{prop}
  [See Theorem 2.29 in \cite{adams_sobolev_2003}]
  \label{prop:property.mollifiers}
  For all $f \in \LintR1$ and all $v \in \mathbb N$,  the function \(\mathcal M_v f\) is continuous. Moreover,
  $$\int_0^1 |\mathcal M_v f(t) - f(t)| {\rm d}t \rightarrow 0 \text{ as } v \rightarrow \infty.$$
  Furthermore, if $f$ is a continuous function of $t$, then 
  $$\sup_{t \in [0, 1]} |\mathcal M_v f(t) - f(t)| \rightarrow 0 \text{ as } v \rightarrow \infty.$$
\end{prop}

\begin{lemma}
  \label{lem:property.mollifiers}For  for any $v \in \mathbb N$, the mollifying operator $\mathcal M_v$ satisfies the following properties:
  \begin{itemize}
  \item[(a)] For any $M \in \SintRnn$, if $M(t) \succeq 0 \; \forall t\in [0, 1] \text{ a.e.}$, then $\mathcal M_v M(t) \succeq 0 \; \forall t\in [0, 1]$.
  \item[(b)] For any $a \in \mathbb R[t]$ and $x \in \LintR1$, $\sup_{t \in [0, 1]}|a(t) \mathcal M_v x(t) - \mathcal M_v (a \cdot x) (t)|  \rightarrow 0$ as $v \rightarrow \infty$.
  \item[(c)] For any polynomial function $d: \mathbb R^2 \rightarrow \mathbb R$ and $x \in \LintR1$, let $g(t) = \int_0^t d(t, s) x(s) {\rm d} s$. Then $$\sup_{t \in [0, 1]} \left|\mathcal M_v g(t) -  \int_0^t d(t, s) \mathcal M_v x(s) {\rm d} s\right| \rightarrow 0 \text{ as } v \rightarrow \infty.$$
  \end{itemize}
\end{lemma}
\begin{proof}
  The proof of (a) simply follows from the fact that the function $J$ is nonnegative on $\mathbb R$.
  
  To prove (b), let $L \coloneqq \sup_{t \in [0, 1]} |a'(t)|$ and $\gamma \coloneqq \sup_{t \in [0, 1]} |x(t)|$. Notice that
  \begin{align*}
    a(t) \mathcal M_v x(t) - \mathcal M_v (a\cdot x) (t)
    & = \int_0^1 v J(v(t-s)) (a(t) - a(s))x(s) {\rm d} s.
  \end{align*}
  Hence,
  $$|a(t) \mathcal M_v x(t) - \mathcal M_v (ax) (t)| \le L \gamma \int_0^1 v J(v(t-s)) |t - s| {\rm d} s.$$
By the change of variable $u \coloneqq v(s-t)$ and in view of the evenness of the function $J,$ we get
  $$\int_0^1 v J(v(t-s)) |t-s|{\rm d}s=  \frac1v \int_{-vt}^{v(1-t)}  J(u) \left| u\right|{\rm d} u \le \frac1v \int_{-1}^1  J(u) {\rm d}u \le \frac 1v.$$
  Therefore,
  $$\sup_{t \in [0, 1]}|a(t) \mathcal M_v x(t) - \mathcal M_v (a\cdot x) (t)| \le \frac{L\gamma}v$$ 
  and the claim follows.
  
  Let us now prove $(c)$. Observe that, on the one hand, for every $t \in [0, 1]$,
  \begin{align*}
    \left|\mathcal M_v g(t) -  \int_0^t d(t, s) \mathcal M_v x(s) \; {\rm d}s \right|
    & \le  \left|\mathcal M_v g(t) - g(t)\right| + \left|g(t) - \int_0^t d(t, s) \mathcal M_v x(s) {\rm d} s\right|
    \\& \le  \left|\mathcal M_v g(t) - g(t)\right| + \int_0^t |d(t, s)| \cdot |x(s) - \mathcal M_v x(s)| \; {\rm d} s
    \\& \le  \left|\mathcal M_v g(t) - g(t)\right| + \sup_{t,s \in [0, 1]} |d(t, s)| \int_0^1  |x(s) - \mathcal M_v x(s)| \; {\rm d} s.
  \end{align*}
	On the other hand, from Proposition \ref{prop:property.mollifiers} (and continuity of $g$), we know that
	 $$\sup_{t \in [0, 1]}|\mathcal M_vg (t)- g(t)| \rightarrow 0 \text{ as } v \rightarrow \infty,$$
	 and 
   $$\int_0^1 |\mathcal M_v x(s)  - x(s)| \; {\rm d}s \rightarrow 0 \text{ as }v \rightarrow \infty.$$
   Combining these three facts, we conclude that
  $$\sup_{t \in [0, 1]} |\mathcal M_v g(t) -  \int_0^t d(t, s) \mathcal M_v x(s) {\rm d} s| \rightarrow 0 \text{ as } v \rightarrow \infty.$$
\end{proof}  
  



Properties (b) and (c) in Lemma \ref{lem:property.mollifiers} give the following corollary.
\begin{cor}
  \label{cor:F.molly.commute}
Let $F$ be as in \eqref{eqn:F} with  $A_0, \ldots, A_n, D_1, \ldots D_n$  polynomial and let $\|\cdot\|$ be any matrix norm. Then, 
  $$\sup_{t \in [0, 1]} \|\mathcal M_v Fx (t)-  F \mathcal M_vx(t)\| \rightarrow 0 \text{ as } v \rightarrow \infty.$$
\end{cor}

We now go back to the proof of optimality of polynomial solutions. The idea is as follows. We know that for any $\varepsilon > 0$ there exists a feasible solution $x^\varepsilon$ whose objective value is \bachir{within} $\varepsilon$ of $\opt$ when $\opt < \infty$ and larger than $\frac1\varepsilon$ when $\opt = +\infty$. 
  We construct a sequence of feasible polynomial solutions whose objective value \bachir{converge} to $\langle c, x^\varepsilon \rangle_\LintRn$. We do so in three steps. First, using existence of a strictly feasible solution \(x^s\) to \eqref{eqn:time_varying_sdp_l2}, we perturb \(x^\varepsilon\) slightly to make it strictly feasible without changing its objective value by much. Second, we approximate the perturbed solution by a continuous solution using mollifying operators. Finally, we invoke the Stone-Weierstrass theorem to approximate the continuous solution with a polynomial solution.

\newcommand{\nearopt}{x^{\varepsilon}}
\newcommand{\stnearopt}{x^{\lambda, \varepsilon}}
\newcommand{\mollnearopt}{x^{v, \lambda, \varepsilon}}
\newcommand{\polynearopt}{p^{s, v, \lambda, \varepsilon}}

\begin{customproof}{of Theorem \ref{thm:poly.optimal}}
For any $\varepsilon > 0$, let $x^\varepsilon$ be a feasible solution to the TV-SDP in \eqref{eqn:time_varying_sdp_l2} such that
		$$\opt-\langle c, x^\varepsilon \rangle_\LintRn \le \varepsilon \text{ if $\opt < \infty$ and } \langle c, x^\varepsilon \rangle_\LintRn \ge \frac1\varepsilon \text{ if $\opt = +\infty$}.$$
  Let $x^s$ be any strictly feasible solution to the TV-SDP in \eqref{eqn:time_varying_sdp_l2} and for \(\lambda \in (0, 1)\) let  $$x^{\lambda,\varepsilon} := (1-\lambda) x^\varepsilon + \lambda x^s.$$
   Observe that for all $\varepsilon > 0$ and $\lambda \in (0, 1)$ the function \(x^{\lambda,\varepsilon}\) is also strictly feasible to \eqref{eqn:time_varying_sdp_l2}. 
   Moreover, as $\lambda  \rightarrow 0$, $\langle c, x^{\lambda, \varepsilon} \rangle_\LintRn \rightarrow \langle c, x^{\varepsilon} \rangle_\LintRn $.
  
  For a nonnegative integer $v$, let $\mathcal M_v$ be the mollifying operator that appears in Definition \ref{def:mollifiers} and Remark \ref{rmk:mollifiers}.
  For all $\varepsilon > 0$, $\lambda \in (0, 1)$, and $v \in \mathbb N$, let \(\mollnearopt \coloneqq \mathcal M_v \stnearopt\). The function \(\mollnearopt\) is  continuous by Proposition \ref{prop:property.mollifiers}.
   We claim that $\mollnearopt$ is strictly feasible to the TV-SDP in \eqref{eqn:time_varying_sdp_l2} for any $\varepsilon > 0$ and $\lambda \in (0, 1)$ when $v$ is large enough.
   Indeed, for any such $\varepsilon$ and $\lambda$, there exists $\beta_{\lambda, \varepsilon} > 0$ such that
   $F   \stnearopt(t) \succeq \beta_{\lambda, \varepsilon} I \; \forall t \in [0, 1] \text{ a.e.}$.
   By property (a) of Lemma \ref{lem:property.mollifiers}, 
   $$\mathcal M_v F   \stnearopt(t) \succeq \beta_{\lambda, \varepsilon} I \; \forall t \in [0, 1].$$
  Let $\|\cdot\|$ be any matrix norm. Using Corollary \ref{cor:F.molly.commute}, 
  $$\sup_{t \in [0, 1]} \| \mathcal M_v F\stnearopt(t) -  F\mollnearopt(t)\| \rightarrow 0 \text{ as } v \rightarrow \infty.$$
  By continuity of the minimum eigenvalue function we conclude that for $v$ high enough, $$F\mollnearopt(t) \succeq \frac{\beta_{\lambda, \varepsilon}}2 I\; \forall t \in [0, 1].$$
	Moreover, for all $\varepsilon > 0$, $\lambda \in (0, 1)$, Proposition \ref{prop:property.mollifiers} implies that
	$$\langle c, \mollnearopt \rangle_\LintRn \rightarrow \langle c, \stnearopt \rangle_\LintRn \text{ as } v \rightarrow \infty.$$ 


 As a final step,  for a fixed $\varepsilon > 0$, $\lambda \in (0, 1)$, and $v \in \mathbb N$, we invoke the Stone-Weierstrass theorem to approximate \(\mollnearopt\) by a sequence $\{\polynearopt\}_{s \in \mathbb N}$ of polynomial elements of $\LintRn$ such that
 $$\sup_{t \in [0, 1]}\|\mollnearopt(t) - \polynearopt(t)\|_{\infty} \rightarrow 0 \text{ as } s \rightarrow \infty.$$

 Note that 
 $$\underset{t \in [0, 1]}\sup \|F\mollnearopt(t) - F\polynearopt(t)\| \le C \underset{t \in [0, 1]}\sup \|\mollnearopt(t) -\polynearopt(t)\|_\infty$$
  where $C \coloneqq \underset{t,s \in [0, 1]}\sup \sum_{i=1}^n\|A_i(t)\|+ \|D_i(s, t)\|.$
 By the same reasoning as before, for $s$ high enough, the polynomial $\polynearopt$ will be (strictly) feasible to our TV-SDP. Moreover, 
 $$\langle c, \polynearopt \rangle_\LintRn \rightarrow \langle c, \mollnearopt \rangle_\LintRn \text{ as } s \rightarrow \infty.$$
 To get the overall result, fix $\varepsilon$ small enough, then $\lambda$ small enough, then $v$ large enough, and then $s$ large enough.
\end{customproof}

\subsection{Finding the best polynomial solution to a TV-SDP via SDP}
\label{sec:tvsdp_is_sdp}

In this section, we show how one can find the best polynomial solution of a given degree to a TV-SDP. This is done by reformulating the problem as a semidefinite program. This formulation is based on the fact that any univariate polynomial matrix \(X(t)\) that is positive semidefinite over an interval has a certain sum of squares representation of low degree. This representation can be found by semidefinite programming using the well-known connection (see e.g. \cite{parrilo_semidefinite_2003, lasserre_global_2001}) between sum of squares polynomials and SDPs.

\newcommand \nummonomials {pm}
\newcommand \psdmatrices[1][pm]{{\mathcal S^{#1}}^+}


\begin{thm}\bachir{\bf (See \cite[Theorem 2.5]{dette_matrix_2002}, \cite[Theorem 6.11]{papp2013semidefinite}, and see~\cite{aylward2007explicit} for a history of related proofs)}
 \label{thm:characeterization.sos.matrix}
Let $X \in \SintRnn_{d}$ be a univariate $m \times m$ polynomial matrix of degree $d$.  If $d$ is odd,
then $X(t) \succeq 0\; \forall t \in [0, 1]$ if and only if there exist (not necessarily square) polynomial matrices $B_1$ and $B_2$  of degree $\frac{d-1}2$ such that
$$X(t) = t B_1(t)^TB_1(t) + (1-t) B_2(t)^TB_2(t).$$
Similarly, if $d$ is even,
then $X(t) \succeq 0\; \forall t \in [0, 1]$ if and only if there exist (not necessarily square) polynomial matrices $B_1$ and $B_2$ of degree $\frac{d}2$ and $\frac{d}2-1$ respectively such that
$$X(t) = B_1(t)^TB_1(t) + t(1-t) B_2(t)^TB_2(t).$$
%
\end{thm}
%

This theorem results in a semidefinite representation of polynomial matrices that are positive semidefinite on the interval $[0, 1]$ as we describe next. \bachir{This transformation is rather standard and can be traced back to the work of Nesterov \cite{nesterov2000squared}.}
  
\begin{prop}
  Let  $d, m$ be positive integers. There exist two linear maps $\alpha_d^m$ (which maps $\mathcal S^{\frac {d+1}2m}$ to  $\SintRnn$ when $d$ is odd and $\mathcal S^{(\frac d2+1)m}$ to  $\SintRnn$ when $d$ is even) and $\beta_d^m$ (which maps $\mathcal S^{\frac {d+1}2m}$ to  $\SintRnn$ when $d$ is odd and  $\mathcal S^{\frac d2m}$ to  $\SintRnn$ when $d$ is even) such that for any $X \in \SintRnn_d$, $X(t) \succeq 0 \; \forall t \in[0, 1]$ if and only if one can find positive semidefinite matrices $Q_1, Q_2$ of appropriate sizes that satisfy the equation
  $$X = \alpha_d^m(Q_1) + \beta_d^m(Q_2).$$
  

\label{prop:positivestellnaz_sdp_finite}
\end{prop}
\begin{proof}
  Fix positive integers $m$ and $d$.
  Let $Y \in \SintRnn_d$ be an $m \times m$ polynomial matrix of degree $d$.
  It is well known that $Y$ can be written as $B(t)^TB(t)$ for some polynomial matrix $B$ if and only if the polynomial $y^TY(t)y$ is a sum of squares of some polynomials in the variables $(t, y_1, \ldots, y_m)$; see e.g. \cite{kojima_sums_2003}. The latter condition is equivalent to existence of a $(\frac d2+1)m \times (\frac d2+1)m$ matrix $Q \succeq 0$ ($d$ is necessarily even) such that the following polynomial identity holds
  \begin{equation}
    \label{eq:psd.matrix.is.sos}
    y^TY(t)y = v(t, y)^T Q v(t, y),
    \end{equation}
    where 
    $$v(t, y) = (y_1, \ldots, y_m, y_1t, \ldots, y_mt, \ldots, y_1 t^{\frac d2}, \ldots, y_m t^{\frac d2} )^T$$
     is the vector of all monomials of the form $y_lt^k$  for $l=1,\ldots,m$, and $k=0,\ldots,\frac d2$; see e.g.  \cite[Section 3]{ahmadi_complete_2013}. 
    For notational convenience, we index the entries of the matrix $Q$ by the monomials in $v \coloneqq v(t, y)$. This means that when we write $Q_{v_i, v_j}$, we refer the $(i, j)\text{-th}$ entry of  $Q$.

    Note that for any symmetric matrix $Q$, there exists a unique $Y \in \SintRnn$ that satisfies the identity \eqref{eq:psd.matrix.is.sos}.
    Indeed, considering the expression $v(t, y)^T Qv(t, y)$ as a polynomial in $y_1, \ldots, y_m$ with coefficients in $\mathbb R[t]$, the coefficient of $y_iy_j$ is equal to twice the $(i,j)\text{-th}$ entry of $Y(t)$ when $i\ne j$, and equal to the $(i,i)\text{-the}$ entry of $Y(t)$ otherwise. Define $\Lambda^m_d$ to be the linear function that maps a symmetric matrix $Q \in {\mathcal S^{(\frac d2+1) m}}$ to the $m \times m$ polynomial matrix $Y$ of degree $d$ that satisfies identity \eqref{eq:psd.matrix.is.sos}, i.e.
\begin{equation}
      \label{eqn:lambda}
    Y = \Lambda^m_d(Q) \iff Y_{ij}(t) = c_{ij} \sum_{k,l \in \{0, \ldots, \frac d2\}} Q_{y_it^k, y_jt^l} t^{k+l} \quad \forall i, j \in \{ 1,\ldots,m\},
\end{equation}
    with $c_{ij} = \frac12$ when $i \ne j$ and $c_{ii} = 1$.
    
    We have just shown that an $m \times m$ polynomial matrix $Y$ of degree $d$ can be written as $Y(t) = B(t)^TB(t)$ for some polynomial matrix $B$ if and only if there exists an $m(\frac{d}2+1) \times m(\frac{d}2+1)$ positive semidefinite matrix $Q$ such that
    $$Y = \Lambda_d^m(Q).$$
    Combining this result with Theorem \ref{thm:characeterization.sos.matrix}, we get that any $m \times m$ polynomial matrix $X$ is positive semidefinite on $[0, 1]$ if and only if there exist positive semidefinite matrices $Q_1, Q_2$ such that
    $$X = \alpha_d^m(Q_1) + \beta_d^m(Q_2),$$
    where
    
    \begin{equation}
      \begin{aligned}
    \alpha_d^m(Q) = t\Lambda_{d-1}^m(Q) \text{ and } \beta_d^m(Q) = (1-t)\Lambda_{d-1}^m(Q) \text{ when $d$ is odd},\\
    \alpha_d^m(Q) = \Lambda_{d}^m(Q) \text{ and } \beta_d^m(Q) = t(1-t)\Lambda_{d-2}^m(Q) \text{ when $d$ is even}.
      \end{aligned}\label{eqn:alpha.beta}
  \end{equation}
\end{proof}

The next theorem summarizes the results of this subsection. 

\begin{thm}
For \(d \in \mathbb N\), the following SDP finds the best polynomial solution of degree \(d\) to the TV-SDP in \eqref{eqn:time_varying_sdp_l2} with data $c, A_0, A_1, \ldots, A_n, D_1,\ldots, D_n$:

\begin{equation}\label{eq:sdp_find_best_poly_d}
\begin{array}{ll@{}ll}
\underset{x, Q_1, Q_2}{\max} & \langle c, x \rangle_{\LintRn} & \\
  \text{s.t.}
                                        &x \in \mathbb R^n_d[t]\\
                                        &Q_1, Q_2 \succeq 0,\\
                                        &Fx = \alpha_{d'}^m(Q_1) + {\beta_{d'}^m}(Q_2).
\end{array}
\end{equation} 
Here, $d'$ is the degree of $Fx$, i.e. $$d' = \max\{\deg(A_0), \max_{i=1,\ldots, n} \deg(A_i)+d, \max_{i=1,\ldots, n} \deg(D_i)+d+1\}.$$
\label{thm:tvsdp_as_sdp}
\end{thm}

From a practical standpoint, a nice feature of this SDP hierarchy is the dimensions of the matrices $Q_1,Q_2$ grow only linearly with $d$.


\begin{remark}For implementation purposes, one does not need to explicitly write out the linear maps $\alpha_{d'}^m$ and $\beta_{d'}^m$ in Theorem \ref{thm:tvsdp_as_sdp}.
  Certain solvers, such as YALMIP \cite{lofberg_yalmip_2004} or SOSTOOLS \cite{papachristodoulou_sostools:_2013}, accept the problem directly in the following form (and do the conversion to an SDP in the background):

\begin{equation*}
  \begin{array}{ll@{}ll}
    \underset{x, X_1, X_2}{\max} & \langle c, x \rangle_{\LintRn} & \\
                                 \text{s.t.}
                                 &x \in \mathbb R^n_d[t], X_1, X_2 \in \SintRnn_{d'-1},\\
                                 &Fx (t) = t X_1(t) + (1-t) X_2(t),\\
                                 & y^TX_1(t)y, y^TX_2(t)y \ \text{are sums of squares of polynomials},
  \end{array}
\end{equation*}
when $d$ is odd, and 
\begin{equation*}
  \begin{array}{ll@{}ll}
    \underset{x, X_1, X_2}{\max} & \langle c, x \rangle_{\LintRn} & \\
                                 \text{s.t.}
                                 &x \in \mathbb R^n_d[t], X_1 \in \SintRnn_{d'}, X_2 \in \SintRnn_{d'-2},\\
                                 &Fx (t) = X_1(t) + t(1-t) X_2(t),\\
                                 & y^TX_1(t)y, y^TX_2(t)y \ \text{are sums of squares of polynomials},
  \end{array}
\end{equation*}
when $d$ is even. The aforementioned linear maps however help us with the notation and presentation of the next section. 
\end{remark}

\begin{remark}
\bachir{
The results of this section combined with Theorem \ref{thm:poly.optimal} imply the following: For any strictly feasible TV-SDP, the SDPs in \eqref{eq:sdp_find_best_poly_d}, indexed by the integer $d$, produce a sequence of \emph{feasible} polynomials to the TV-SDP in \eqref{eqn:time_varying_sdp_l2}, whose objective values converge to the optimal value of \eqref{eqn:time_varying_sdp_l2}. 
 Note that we are not making any claims about the convergence
    of the sequence of polynomials returned by these SDPs. Indeed, our interest is not for these polynomials to be close (in some distance measure) to an optimal solution of \eqref{eqn:time_varying_sdp_l2}  (which might not even be continuous), but for them to be feasible and have arbitrarily good objective value.
    }
\end{remark}

\section{The Dual Approach: Obtaining Upper Bounds}
\label{sec:dual-approach}
In the previous section, we showed how one can obtain arbitrarily good lower bounds on the optimal value of a strictly feasible TV-SDP by searching for polynomial solutions of increasing degree. In practice, one often has a computational budget and cannot increase the degree of the candidate polynomials beyond a certain threshold. It is therefore very valuable to know how far the objective value of the best polynomial solution that one has found is from the optimal value of \eqref{eqn:time_varying_sdp_l2}. Addressing this need is the subject of this section. More specifically, in Section \ref{sec:weak.and.strong.duality} below we give a hierarchy of dual problems (or relaxations), indexed by a nonnegative integer $d$, that provide a sequence of improving upper bounds on the optimal value of \eqref{eqn:time_varying_sdp_l2}. We show that under the boundedness assumption in \eqref{eqn:assump.boundedness}, these upper bounds converge to the optimal value as $d \rightarrow \infty$.
While the original formulation of the dual problems in Section \ref{sec:weak.and.strong.duality} is infinite dimensional, we show in Section \ref{sec:dual.is.sdp} that each of these problems can be solved exactly as an SDP of tractable size.

\bachir{We remark that our dual problems are different to the best of our knowledge from those appearing in the literature on continuous linear programs. For instance, we can derive the Lagrangian dual of problem \eqref{eqn:time_varying_sdp_l2} using standard techniques from infinite-dimensional linear programming \cite{anderson1987linear}. This problem would read

\begin{equation}
\label{eqn:lagrange_dual_time_varying_sdp_l2}
\begin{array}{ll@{}ll}
\underset{P \in \SintRnn^+([0, 1])}{\inf} & \langle P, A_0 \rangle_{\SintRnn} & \\
  \text{subject to}&  \tilde F^*P(t) + c(t) = 0  &\; \talmostsurely,\\
\end{array}
\end{equation}
where $\tilde F^*P(t)$ is the vector of size $n$ whose $i\text{-th}$ component is given by the $(i+1)\text{-th}$ component of $F^*P(t)$, with $F^*$ as in (\ref{eqn:Fstar}) below.
This is of course another TV-SDP, for which we can search for polynomial solutions to obtain upper bounds on the optimal value of \eqref{eqn:time_varying_sdp_l2}.}
The reason we do not take this approach is that we do not want to make strict feasibility assumptions on both the primal and the dual (which our current proof strategy would require in order to ensure convergence of these bounds). Furthermore, more involved assumptions would likely be required to guarantee strong duality between the two TV-SDPs. Even in the special case of continuous linear programs, a number of assumptions are needed to obtain strong duality \cite{grinold_continuous_1969, bampou_polynomial_2012}.

\newcommand{\adjoint}[1]{{#1}^*}

The following definitions will be useful in the formulation of our dual problems.

\begin{definition}[Adjoint maps]
\label{def:adjoints}
We define the adjoint of the affine map $F$ in (\ref{eqn:F}) to be the linear map $F^*\colon \SintRnn  \to \LintR{n+1}$ that acts on $P \in \SintRnn$ as follows:
  \begin{equation}
    \label{eqn:Fstar}
    F^*P(t) =   \small\begin{pmatrix}\Tr(A_0(t) P(t))\\ \Tr( A_1(t) P(t) ) + \int_t^1 \Tr (D_1(t, s) P(s)) \;{\rm d}s \\\vdots\\ \Tr( A_n(t) P(t)) + \int_t^1 \Tr (D_n(t, s) P(s)) \;{\rm d}s\end{pmatrix}\normalsize.
  \end{equation}

    For an even positive integer $d$, the adjoint of the linear map $\Lambda^m_d$ defined in \eqref{eqn:lambda} is the linear map $\adjoint{\Lambda^m_d}\colon \SintRnn  \to \mathcal S^{(\frac d2+1) m}$ that acts on $P \in \SintRnn$ as follows:
$$ \adjoint{\Lambda^m_d} (P) \coloneqq \left(\int_0^1 t^{k+l} P_{ij}(t) {\rm d}t\right)_{y_it^k,y_jt^l},$$
where the notation $Q_{y_it^k,y_jt^l}$ stands for the $(r, s)$ entry of the matrix $Q \in S^{(\frac d2+1) m}$ with $r$ and $s$ being the position of the monomials $y_it^k$ and $y_jt^l$ in the vector
$$(y_1, \ldots, y_m, y_1t, \ldots, y_mt, \ldots, y_1 t^{\frac d2}, \ldots, y_m t^{\frac d2} )^T.$$
For $d \in \mathbb N$, the adjoints of the linear maps $\alpha_d^m$ and $\beta_d^m$ defined in \eqref{eqn:alpha.beta} are defined as follows:
$$\adjoint{\alpha_d^m}(P) \coloneqq   \adjoint{\Lambda_{d-1}^m}(t P) \text{ and } \adjoint{\beta_d^m}(P) \coloneqq \adjoint{\Lambda_{d-1}^m}((1-t) P) \; \forall P \in \SintRnn \text{ when $d$ is odd},$$

$$\adjoint{\alpha_d^m}(P) \coloneqq   \adjoint{\Lambda_{d}^m}(P) \text{ and } \adjoint{\beta_d^m}(P) \coloneqq \adjoint{\Lambda_{d-2}^m}(t(1-t) P) \; \forall P \in \SintRnn \text{ when $d$ is even}.$$
\end{definition}

The reader can check (using Fubini's double integration theorem when needed) that these adjoint maps satisfy the following identities.
\begin{prop}
  \label{prop:prop.adjoint}
  For all $x \in \LintRn$ and $P \in \SintRnn$, 
  $$\langle Fx, P \rangle_\SintRnn = \langle \begin{pmatrix}1\\x\end{pmatrix}, F^*P \rangle_{\LintR{n+1}}.$$

  For a linear map $L \in \{\Lambda_d^m, \alpha_d^m, \beta_d^m\}$, a polynomial matrix $P \in \SintRnn$, and a constant symmetric matrix $Q$ of appropriate size,
  $$\langle {L}(Q), P \rangle_\SintRnn= \Tr( Q L^*(P)).$$
\end{prop}

\subsection{Dual formulation}
\newcommand\diag[1]{\begin{pmatrix}#1_1 & 0 & 0\\0&\ddots&0\\0 & 0 & #1_n\end{pmatrix}}

To derive our dual problems, we start by observing that using Lemma \ref{lem:s.plus.self.dual}, we can rewrite the TV-SDP in \eqref{eqn:time_varying_sdp_l2} as the following problem:
\begin{equation*}
 \begin{array}{lll@{}ll}
  \underset{x \in \LintRn}{\max} & \langle c, x \rangle_\LintRn & \\
  \text{subject to}& \langle Fx, P \rangle_{\SintRnn} \ge 0 \quad \forall P \in \SintRnn^+([0, 1]) \cap \mathbb R^{m \times m}[t].
\end{array}
\end{equation*}
To get an upper bound on the optimal value of \eqref{eqn:time_varying_sdp_l2}, we relax the constraint in this problem
by asking it to hold only for all $P \in \SintRnn^+([0, 1])_d$, i.e. for all polynomial matrices of degree bounded by some threshold $d$. This gives us our dual problem at level $d$, whose optimal value we denote by $u_d$:
\begin{equation}
\label{eqn:ud.tvsdp}
 \begin{array}{lll@{}ll}
  u_d &\coloneqq \underset{x \in \LintRn}{\max} & \langle c, x \rangle_\LintRn & \\
  &\text{subject to}& \langle Fx, P \rangle_{\SintRnn} \ge 0 \quad \forall P \in \SintRnn^+([0, 1])_d.
\end{array}
\end{equation}

\label{sec:weak.and.strong.duality}
\begin{lemma}[Weak Duality]
  \label{lem:weak.duality}
Let $\opt$ denote the optimal value of the TV-SDP in \eqref{eqn:time_varying_sdp_l2} and $u_d$ be as in \eqref{eqn:ud.tvsdp}. Then, for all $d \in \mathbb N$, we have
$$\opt \le u_{d} \text{ and } u_{d+1} \le u_d.$$
\end{lemma}
\begin{proof}
  Fix $d \in \mathbb N$. Since $\SintRnn^+([0, 1])_d \subseteq \SintRnn^+([0, 1])_{d+1}$, it is clear that $u_{d+1} \le u_{d}$.
  Note that the only difference between problem \eqref{eqn:ud.tvsdp} and the TV-SDP in \eqref{eqn:time_varying_sdp_l2} is that we have replaced the constraint $Fx \in \SintRnn^+$ by $\langle Fx, P \rangle_{\SintRnn} \ge 0 \quad \forall P \in \SintRnn^+([0, 1])_d$. By Lemma \ref{lem:s.plus.self.dual}, the former constraint is stronger than the latter. Therefore, $\opt \le u_d$.
\end{proof}

To get strong duality, we will make the additional assumption in (\ref{eqn:assump.boundedness}); i.e. we assume that there exists a positive scalar $\gamma$ such that
  \begin{equation*}
    Fx \in \SintRnn^+ \implies \|x(t)\|_{\infty} < \gamma \; \forall t \in [0, 1] \text{ a.e.}.
  \end{equation*}
  We further require that this constraint already be included in $F$. In other words, $F$ is taken to be of the form
  \begin{equation}
    \label{eqn:F.form}
Fx \coloneqq \begin{pmatrix}
    \hat Fx & 0 & 0\\
    0 & \gamma I - diag(x) & 0\\
    0 & 0 & diag(x)-\gamma I \\
  \end{pmatrix},
\end{equation}
where $\hat Fx \succeq 0$ denotes the remaining constraints of the TV-SDP.

  
  

\begin{thm}[Strong Duality]
  \label{thm:strong.duality}
  Suppose that the TV-SDP in \eqref{eqn:time_varying_sdp_l2} satisfies the boundedness assumption (\ref{eqn:assump.boundedness}) as explicitly imposed by a map $F$ of the form \eqref{eqn:F.form}. Let $\opt \in \mathbb R \cup \{-\infty\}$ denote the optimal value of this TV-SDP. Then the optimal value $u_d$ of problem \eqref{eqn:ud.tvsdp} converges to $\opt$ as $d \rightarrow \infty$.
\end{thm}
\begin{proof} 
From Lemma \ref{lem:weak.duality}, the sequence $\{u_d\}$ is nonincreasing and bounded below by $\opt$. It therefore converges to a (possibly infinite) limit $u^* \ge \opt$. To conclude the proof, we show that $u^* \le \opt$. 
 
 Observe first that if there exists a nonnegative integer $d$ such that $u_d = -\infty$, then $u^* =\opt= -\infty$ and we are done. We can therefore suppose that the sequence $\{u_d\}$ never takes the value $-\infty$. We claim that when $d \ge \deg(c)$,  $u_d$ cannot take the value $+\infty$ either. To see why, fix $d \ge \deg(c)$ and let $x \in \LintRn$ be any function that satisfies
  $\langle Fx, P \rangle_{\SintRnn} \ge 0 \; \forall P \in \SintRnn^+([0, 1])_d$. For any $q \in \mathbb R_d^n[t]$ that is elementwise nonnegative on $[0, 1]$, by taking
  \[P \in \left\{  \begin{pmatrix}
      0 & 0 & 0\\
      0 & diag(q) & 0\\
      0 & 0 &0 \\
    \end{pmatrix} ,  \begin{pmatrix}
      0 & 0 & 0\\
      0 & 0 & 0\\
      0 & 0 & diag(q) \\
    \end{pmatrix}\right\},\]
we get that $|\langle q, x \rangle_\LintRn| \le \gamma \langle q, \vec 1 \rangle_\LintRn$. Let 
$$m_c \coloneqq \min_{i=1,\ldots,n} \min_{t\in [0, 1] } c_i(t)$$
and observe that the polynomial $p(t) \coloneqq c(t) - m_c \vec 1$ is elementwise nonnegative on $[0, 1]$. We therefore have
\begin{align*}
|\langle c, x \rangle_\LintRn| &\le |\langle c-m_c\vec 1, x \rangle_\LintRn| + |\langle m_c\vec 1, x \rangle_\LintRn| 
\\&= |\langle p, x \rangle_\LintRn| + |m_c| |\langle \vec 1, x \rangle_\LintRn|
\\&\le \gamma \langle p, \vec 1 \rangle_\LintRn +  |m_c| \gamma \langle \vec 1, \vec 1 \rangle_\LintRn,
\end{align*}
which proves our claim that $u_d$ is finite.

As a consequence, for any integer $d \ge \deg(c)$ and for any positive scalar $\varepsilon,$ there exists a function $x^{\varepsilon, d} \in \LintRn$ such that
  \begin{equation}
    \label{eq:wd.near.optimal}
    \langle F x^{\varepsilon, d}, P \rangle_{\SintRnn} \ge 0 \; \forall P \in \SintRnn^+([0, 1])_d \text{ and } \langle c, x^{\varepsilon, d} \rangle_\LintRn \ge u_d - \varepsilon.
  \end{equation}
  
  For a given $\varepsilon > 0$, fix such a sequence $\{x^{\varepsilon, d}\}$ indexed by $d$. We claim that $\{x^{\varepsilon, d}\}$ must have a subsequence that converges weakly to a fucntion $x^{\varepsilon}\in\LintRn$.
  %
  %
  Indeed, if we let $\mu^{\varepsilon,d} \coloneqq \frac{l_{x^{\varepsilon, d}}}{\gamma}$, then for any polynomial $p \in \mathbb R^n_d[t]$ that is elementwise nonnegative on $[0, 1]$, we have $|\mu^{\varepsilon,d}(p)| \le \sum_{i=1}^n \int_0^1 p_i(t) \; \rm{d t}.$ Our claim then follows from Theorem \ref{thm:bounded.measure.compact}.
  %
  It is clear by weak convergence that $\langle c, x^\varepsilon \rangle_{\LintRn} \ge u^*-\varepsilon$. Moreover, for any $P \in \SintRnn^+([0, 1]) \cap \mathbb R^{m \times m}[t]$, if $d \ge \deg(P)$, we have $\langle Fx^{\varepsilon,d}, P \rangle_\SintRnn \ge 0$, or equivalently $$\big\langle \begin{pmatrix} 1\\x^{\varepsilon,d}\end{pmatrix}, F^*P \big\rangle_{\LintR{n+1}} \ge 0.$$ Hence, by taking $d \rightarrow \infty$,   $\big\langle \begin{pmatrix} 1\\x^{\varepsilon}\end{pmatrix}, F^*P \big\rangle_{\LintR{n+1}} \ge 0$, showing that $\langle {Fx^\varepsilon}, P \rangle_\SintRnn \ge 0$. By Lemma \ref{lem:s.plus.self.dual}, we conclude that $Fx^\varepsilon \in \SintRnn^+([0, 1])$ and therefore $\langle c, x^\varepsilon \rangle_\LintRn \le \opt$.
  
    We have just proven that for any $\varepsilon > 0$, there exists a feasible solution $x^\varepsilon$ to the TV-SDP in \eqref{eqn:time_varying_sdp_l2} such that $$u^* - \varepsilon \le \langle c, x^\varepsilon \rangle_\LintRn \le \opt.$$ This means that $u^* \le \opt.$
  \end{proof}

\subsection{The dual problem is an SDP}
\label{sec:dual.is.sdp}
In this section, we show that the infinite-dimensional problem in \eqref{eqn:ud.tvsdp} can be converted to an SDP of tractable size.
\begin{thm}
  \label{thm:dual.is.sdp}
  Consider problem \eqref{eqn:ud.tvsdp} at level $d$ and with data $c, A_0, \ldots, A_n, D_1, \ldots, D_n$.
  Let $$\hat d \coloneqq \max\{\deg(c),  \max_{i=1,\ldots,n} d+\deg(A_i),  \max_{i=1,\ldots,n} d+1+\deg(D_i)\}.$$
The optimal value of problem \eqref{eqn:ud.tvsdp} does not change when the space $\LintRn$ is replaced with $\mathbb R^n_{\hat d}[t].$ Moreover, this optimal value is equal to the optimal value of the following SDP
\begin{equation}
\label{eqn:dual.tvsdp}
 \begin{array}{ll@{}ll}
  \underset{x \in \mathbb R^n_{\hat d}[t]}{\max} & \langle c, x \rangle_\LintRn & \\
 \text{subject to}& \adjoint{\alpha_d^m} (Fx) \succeq 0\\&  \adjoint{\beta_d^m} (Fx) \succeq  0,
\end{array}
\end{equation}
where the adjoint maps $\adjoint{\alpha_d^m}, \adjoint{\beta_d^m}$ are as in Definition  \ref{def:adjoints}.
\end{thm}


Just like our primal hierarchy, observe that the dimensions of the matrices that need to be positive semidefinite in this SDP hierarchy grow only linearly with $d$. We start with a simple \bachir{and standard} lemma that will help us prove the first claim of the theorem.

\begin{lemma}
  \label{lem:poly.matching.moments}
 For any function $f \in \LintR{1}$, there exists a sequence of polynomials $\{p_d\}$ such that for every $d \in \mathbb N$, the polynomial $p_d$ has degree $d$ and satisfies
 \begin{equation}
 \int_0^1 q(t) p_d(t) \; {\rm d}t = \int_0^1 q(t) f(t) \; {\rm d}t \quad \forall q \in \mathbb R_d[t].\label{eq:poly.match.moments}
\end{equation}

\end{lemma}

\begin{proof}
  Fix $d \in \mathbb N$.
  Parameterize a generic univariate polynomial $p(t)$ of degree $d$ as
  $$p(t) = \sum_{i=0}^d p_i t^i$$
  and let $m_i \coloneqq \int_0^1 t^i f(t) \; {\rm d}t$ for $i=0, \ldots, d$.
  By linearity, the equality in \eqref{eq:poly.match.moments} is equivalent to
  $$\int_0^1 t^i p(t) \; {\rm d}t = m_i \quad i=0, \ldots, d.$$
  Let $H$ denote the $(d+1)\times(d+1)$ matrix whose $(i,j)\text{-th}$ entry is equal to $\int_0^1 t^{i+j-2}\;{\rm d}t = \frac1{i+j-1}$. Equation \eqref{eq:poly.match.moments} is therefore equivalent to
    $$\begin{pmatrix}p_0,&\ldots&,p_d\end{pmatrix}H  = \begin{pmatrix}m_0,&\ldots,&m_d\end{pmatrix}.$$
    It follows that this equation has a (unique) solution as the matrix $H$ (often named the \emph{Hilbert} matrix) is known to be invertible \cite{hilbert_beitrag_1894}. 
\end{proof}

\begin{customproof}{of Theorem \ref{thm:dual.is.sdp}}
  Fix $d\in\mathbb{N}$.
  Let $x \in \LintRn$ be a feasible solution to (\ref{eqn:ud.tvsdp}), i.e. satisfy
  \begin{equation}
  \langle Fx, P \rangle_\SintRnn \ge 0 \; \forall P \in \SintRnn^+([0, 1])_d.\label{eqn:feasible}
\end{equation}

  Notice that this expression depends on $x$ only through its $\hat d$ moments. More precisely, if a function $y \in \LintRn$ satisfies
  \begin{equation}\langle q, x \rangle_\LintRn  = \langle q, y \rangle_\LintRn \; \forall q \in \mathbb R_{\hat d}^n[t],\label{eqn:match.moments.x}\end{equation}
  then for all $P \in \SintRnn^+([0, 1])_d$,
  $$\langle Fy, P \rangle_\SintRnn = \langle \begin{pmatrix}1\\y\end{pmatrix}, F^*P \rangle_{\LintR{n+1}} = \langle \begin{pmatrix}1\\x\end{pmatrix}, F^*P \rangle_{\LintR{n+1}} = \langle Fx, P \rangle_\SintRnn \ge 0.$$
  Furthermore,  $\langle c, y\rangle_\LintRn =\langle c, x\rangle_\LintRn$.
  By Lemma \ref{lem:poly.matching.moments}, there always exists a function $y$ in $\mathbb R_{\hat d}^n[t]$ that satisfies \eqref{eqn:match.moments.x}. Therefore, we can restrict the space $\LintRn$ in problem \eqref{eqn:ud.tvsdp} to $\mathbb R^n_{\hat d}[t]$.
  Now if $x \in \mathbb R^n_{\hat d}[t]$, by Proposition \ref{prop:positivestellnaz_sdp_finite}, condition \eqref{eqn:feasible} is equivalent to
  $$\langle Fx, \alpha_d^m(Q_1) + \beta_d^m(Q_2) \rangle_\SintRnn \ge 0 \; \forall Q_1 \succeq 0, \forall Q_2 \succeq 0,$$
  which itself is equivalent to
  $$\langle Fx, \alpha_d^m(Q_1) \rangle_\SintRnn \ge 0\; \forall Q_1 \succeq 0, \text{ and } \langle Fx, \beta_d^m(Q_2) \rangle_\SintRnn \ge 0 \; \forall Q_2 \succeq 0.$$
  By Proposition \ref{prop:prop.adjoint}, this latter statement holds if and only if
  $$\langle \adjoint{\alpha_d^m}(Fx), Q_1 \rangle \ge 0\; \forall Q_1 \succeq 0, \text{ and } \langle \adjoint{\beta_d^m}(Fx), Q_2 \rangle \ge 0 \; \forall Q_2 \succeq 0,$$
 i.e.,
  $$\adjoint{\alpha_d^m}(Fx) \succeq 0 \text{ and } \adjoint{\beta_d^m}(Fx) \succeq 0.$$
\end{customproof}

\section{Applications}
\label{sec:applications}
In this section, we present three applications of time-varying semidefinite programs along with some numerical experiments.


\subsection{Time-varying Max-Flow}

\newcommand{\bderiv}{ b^{\text{deriv}}}
\newcommand{\binteg}{ b^{\text{cum}}}
In our first example, we study a generalization of the classical maximum-flow problem where the pipeline capacities are allowed to vary with time. More specifically, we are given a graph with  node set \(V \coloneqq \{1, \ldots n\}\), and edge set  \(E \subseteq [n]^2\). We take node $1$ to be the source of the flow and node \(n\) to be the target.
Our decision variables are functions $f_{ij} \in \LintR{1}$, for $(i, j) \in E$, with $f_{ij}(t)$ denoting the instantaneous flow on edge $(i ,j)$ at time $t \in [0, 1]$.
We have as input functions $b_{i,j} \in \mathbb R[t]$ with $b_{ij}(t)$ denoting the capacity of edge  $(i ,j)$ at time $t \in [0, 1]$.

The capacity (and nonnegativity) constraints that we need to satisfy are
$$0 \le f_{ij}(t) \le b_{ij}(t)  \quad \forall (i, j) \in E, \; \forall t \in [0, 1] \text{ a.e.}.$$
We further need to satisfy conservation of flow constraints at every node other than the source and the target nodes:
$$\underset{j: (i,j) \in E}{\sum} f_{ij}(t) - \underset{j: (j,i) \in E}{\sum} f_{ji}(t) = 0 \quad \forall i \in V \setminus \{1, n\}, \;\forall t \in [0, 1] \text{ a.e.}.$$

In some applications, a subset of the edges that we denote by $E_1 \subseteq E$ may not be able to handle an instantaneous change  in the flow that is too large. In other words, we need to impose the following additional constraints:
\begin{equation}
\left|\frac{{\rm d}}{{\rm d} t}f_{ij}(t)\right| \le \bderiv_{ij}(t) \quad \forall (i, j) \in E_1, \; \forall t \in [0, 1] \text{ a.e.},\label{eqn:maxflow.deriv.constraint}
\end{equation}
for some pre-specified functions $\bderiv_{ij} \in \mathbb R[t]$. We handle this by introducing a new decision variable $g_{ij} \in \LintR{1}$ for every $(i, j) \in E_1$ and imposing
$$\int_0^t g_{ij}(s) \; {\rm d}s - f_{ij}(t) = 0 \text{ and } -\bderiv_{ij}(t) \le g_{ij}(t)\le \bderiv_{ij}(t) \quad  \forall (i, j) \in E_1,\; \forall t \in [0, 1] \text{ a.e.}.$$
Moreover, we assume that because of limitations on production of the flow at the source node, the cumulative flow going into the network up to time $t$ cannot exceed $\binteg(t)$ for some given function $\binteg \in \mathbb R[t]$. Hence, this \bachir{constraint} reads
\begin{equation}
\int_0^t \underset{(1, j) \in E}{\sum} f_{1j}(t) dt \le \binteg(t) \quad \forall t \in [0, 1] \text{ a.e.}.
\label{eqn:maxflow.cum.constraint}
\end{equation}

%

Our objective is to send as much flow as possible from the source to the target node over the time interval $[0, 1]$. Hence, the overall problem, which is a time-varying semidefinite (in fact, linear) program, reads:

\newcommand{\dt}{{\rm d}t}
\newcommand{\ds}{{\rm d}s}
\begin{equation}
\label{eq:tv.maxflow}
\begin{array}{ll}
  \underset{f_{ij}, g_{ij}}{\max}
  & \int_{0}^1  \underset{(1,j) \in E}{\sum} f_{1j}(t) \; \dt \\
  & \left.\begin{array}{ll}
   0 \le f_{ij}(t) \le b_{ij}(t)  & \forall (i, j) \in E\\
   \underset{j: (i,j) \in E}{\sum} f_{ij}(t) -  \underset{j: (j,i) \in E}{\sum} f_{ji}(t) = 0 & \forall i \in V \setminus \{1, n\}\\
   \int_0^t g_{ij}(s) \; {\rm d}s - f_{ij}(t) = 0  &  \forall (i, j) \in E_1\\
            -\bderiv(t) \le g_{ij}(t)\le \bderiv(t) & \forall (i, j) \in E_1\\
  \int_0^t \underset{(1,j) \in E}{\sum} f_{1j}(s) \; \ds \le \binteg(t)\\
    \end{array}\right\}\forall t \in [0, 1] \text{ a.e.}.
\end{array}
\end{equation}

\newcommand \widthfigure{0.1\textwidth}
\newcommand \picforedge[2]{\includegraphics[trim={4cm 2cm 4cm 3cm}, clip, width=\widthfigure]{includes/flow-pngs/flow-#1-#2.png}}
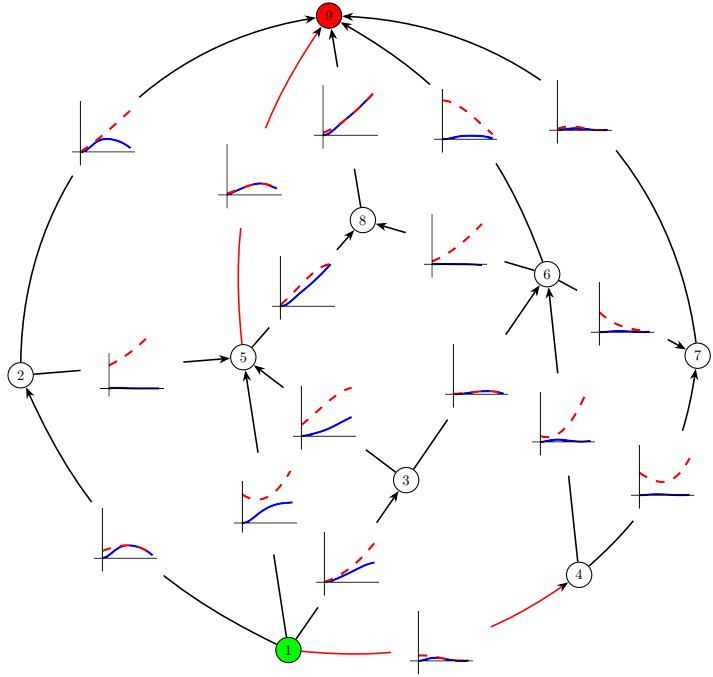
\begin{figure}[h!]
  \centering
  \scalebox{.5}{
    \hspace*{-.75in}
    \begin{tikzpicture}\begin{scope}[every node/.style={circle,thick,draw}]
        \node[fill=green] (1) at (5.12,-1.00) {1};
        \node (2) at (-2.0,6.31) {2};
        \node (3) at (8.25,3.52) {3};
        \node (4) at (12.85,1.) {4};
        \node (5) at (3.92,6.79) {5};
        \node (6) at (12, 9) {6};
        \node (7) at (16,6.83) {7};
        \node (8) at (7.10,10.44) {8};
        \node[fill=red] (9) at (6.20,15.88) {9};
      \end{scope}\begin{scope}[>={Stealth[black]},
        every node/.style={fill=white,circle},
        every edge/.style={draw=black,very thick}]
        \path [->] (1) edge[bend left=20] node {\picforedge{1}{2}} (2);
        \path [->] (1) edge node {\picforedge{1}{3}} (3);
        \path [->] (1) edge[red, bend right=20] node {\picforedge{1}{4}} (4);
        \path [->] (1) edge node {\picforedge{1}{5}} (5);
        \path [->] (2) edge node {\picforedge{2}{5}} (5);
        \path [->] (2) edge[bend left=40] node {\picforedge{2}{9}} (9);
        \path [->] (3) edge node {\picforedge{3}{5}} (5);
        \path [->] (3) edge node {\picforedge{3}{6}} (6);
        \path [->] (4) edge node {\picforedge{4}{6}} (6);
        \path [->] (4) edge[bend right=20] node {\picforedge{4}{7}} (7);
        \path [->] (5) edge node {\picforedge{5}{8}} (8);
        \path [->] (5) edge[red, bend left=20] node {\picforedge{5}{9}} (9);
        \path [->] (6) edge node {\picforedge{6}{7}} (7);
        \path [->] (6) edge node {\picforedge{6}{8}} (8);
        \path [->] (6) edge[bend right=20] node {\picforedge{6}{9}} (9);
        \path [->] (7) edge[bend right=40] node {\picforedge{7}{9}} (9);
        \path [->] (8) edge node {\picforedge{8}{9}} (9);
      \end{scope}\end{tikzpicture}
  }
  \caption{\label{fig:maxflow-graph} An instance of the time-varying max-flow problem. The edge capacities $b_{ij}(t)$, over the time interval $[0, 1]$, are plotted with red dotted lines. The optimal polynomial flow $f_{ij}(t)$ of degree at most $10$ is plotted on each edge with solid blue lines.}
\end{figure}


As a numerical example, we consider the network in Figure \ref{fig:maxflow-graph} with capacities $b_{ij}(t)$ plotted with red dotted lines on each edge $(i ,j)$. Each of these polynomials $b_{ij}$ is a nonnegative polynomial of degree $3$ that is generated as follows
\bachir{
\begin{equation}
\bachir{b_{ij}(t) = t (a_{ij}^{(1)} + a_{ij}^{(2)}\;t)^2 + (1-t)(a_{ij}^{(3)} + a_{ij}^{(4)}\;t)^2,}\label{eq:random_b_ij}
\end{equation}
where $a_{ij}^{(k)}$ are generated independently and uniformly at random from $[-1, 1]$. 
} 
We take $E_1~=~\{(1, 4), (5, 9)\}$, $\bderiv(t) = \frac12$, and $\binteg(t)=t^2$.

Using the machinery of Section \ref{sec:primal-approach}, we solve semidefinite programs (as given in Theorem \ref{thm:tvsdp_as_sdp}) that find the best polynomial solution of degree $d \in \{2, 3, \ldots, 10\}$ to the TV-SDP in \eqref{eq:tv.maxflow}. The optimal values of these problems, which provide improving lower bounds on the optimal value of problem \eqref{eq:tv.maxflow}, are reported in \bachir{the first row of Table \ref{tbl:maxflow.table}.}
We also plot the best polynomial solution of degree $10$ on each edge of the graph in Figure \ref{fig:maxflow-graph} with solid blue lines. Figure \ref{fig:maxflow.constraints} shows that this solution satisfies the constraints in \eqref{eqn:maxflow.deriv.constraint} and \eqref{eqn:maxflow.cum.constraint}.


\begin{figure}[h]
  \begin{subfigure}{.5\textwidth}
    \includegraphics[width=.9\linewidth]{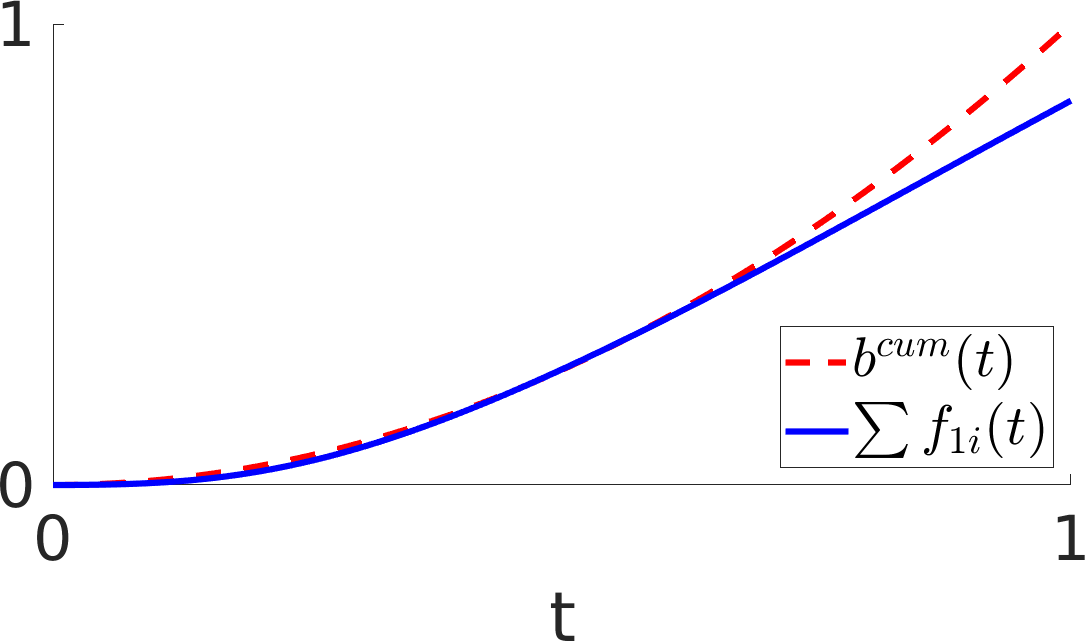}
    \caption{The cumulative flow $\sum_{(1,i) \in E} f_{1i}(t)$ at time $t$ going through the network and the maximum flow available $\binteg(t)$ up to that time.}
  \end{subfigure}\quad\quad\quad\quad
  \begin{subfigure}{.5\linewidth}
    \includegraphics[width=.9\linewidth]{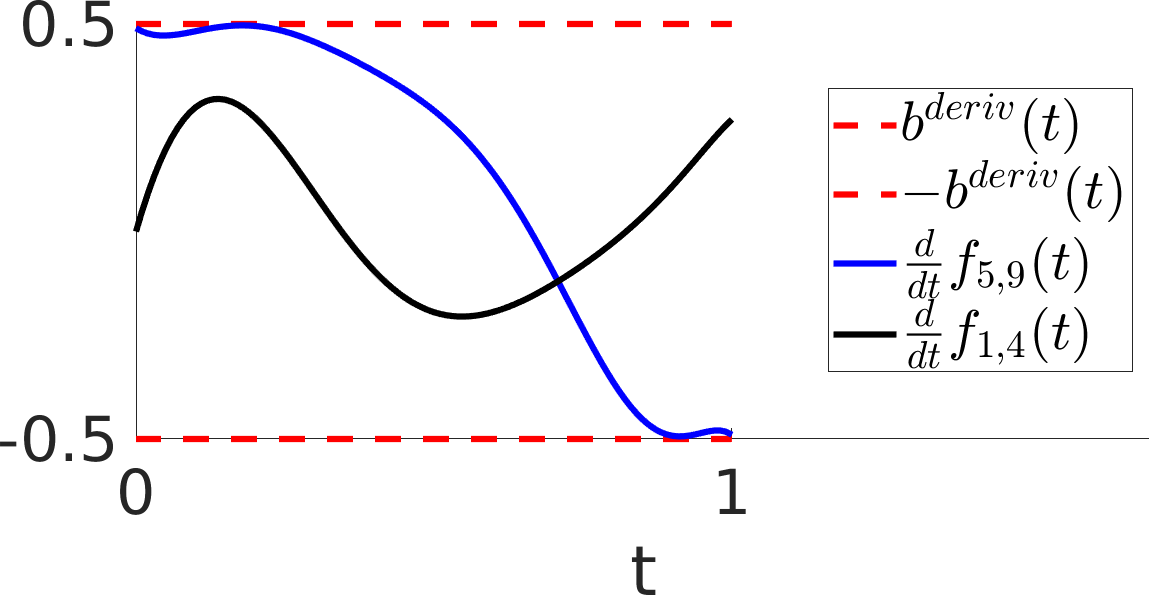}
    \caption{The derivative of the flow going through the edges in $E_1$ and the maximum rate of change $\bderiv(t)$ and $-\bderiv(t)$ allowed for the flow.}
  \end{subfigure}
  \caption{\label{fig:maxflow.constraints} Plots demonstrating that the constraints in \eqref{eqn:maxflow.deriv.constraint} and \eqref{eqn:maxflow.cum.constraint} are satisfied by the best polynomial solution of degree $10$ for \eqref{eq:tv.maxflow}.}
\end{figure}

We also use the machinery of Section \ref{sec:dual-approach} to solve the dual problems in \eqref{eqn:ud.tvsdp} in order to get upper bounds on the optimal value of the TV-SDP in \eqref{eq:tv.maxflow}. By Theorem \ref{thm:dual.is.sdp}, the dual problem at level $d$ is equivalent (after some rewriting) to the following SDP:
\newcommand{\adjalpha}{\adjoint{\alpha_d^1}}
\newcommand{\adjbeta}{\adjoint{\beta_d^1}}
\begin{equation}
  \label{eq:tv.maxflow.dual}
  \begin{array}{ll}
  \underset{f_{ij}, g_{ij} \in \mathbb R_{d+1}[t]}{\max}
  & \int_{0}^1  \underset{(1,j) \in E}{\sum} f_{1j}(t) dt \\
  & \begin{array}{llr}
      \adjalpha (b_{ij}- f_{ij}) \succeq 0, &
      \adjbeta (b_{ij}- f_{ij}) \succeq 0
      & \forall (i, j) \in E\\
      \adjalpha  (f_{ij}) \succeq 0, &
      \adjbeta  f_{ij} \succeq 0
      & \forall (i, j) \in E\\
      \adjalpha \left(\underset{j: (i,j) \in E}{\sum} f_{i j} -  \underset{j: (j,i) \in E}{\sum}f_{j i}\right) = 0, &
      \adjbeta \left(\underset{j: (i,j) \in E}{\sum} f_{i j} -  \underset{j: (j,i) \in E}{\sum} f_{j i}\right) = 0
      & \forall i \in V \setminus \{1, n\}\\
      \adjalpha \left(\int_0^t g_{ij}(s) \; {\rm d}s - f_{ij}\right) = 0, &
      \adjbeta \left(\int_0^t g_{ij}(s) \; {\rm d}s - f_{ij}\right) = 0
      & \forall (i, j) \in E_1\\
      \adjalpha \left(\bderiv-g_{ij}\right) \succeq 0, &
      \adjbeta \left(\bderiv-g_{ij}\right) \succeq 0
      & \forall (i, j) \in E_1\\
      \adjalpha \left(\bderiv+g_{ij}\right) \succeq 0, &
      \adjbeta \left(\bderiv+g_{ij}\right) \succeq 0
      & \forall (i, j) \in E_1\\
      \adjalpha \left(\binteg - \int_0^t\underset{{(1, j) \in E}}{\sum}f_{1j}(s)\right) \succeq 0, &
      \adjbeta \left(\binteg - \int_0^t \underset{{(1, j) \in E}}{\sum} f_{1j}(s)\right) \succeq 0.
    \end{array}
  \end{array}
\end{equation}
%

The optimal value of this problem for different values of $d$ is reported in \bachir{the second row of Table~\ref{tbl:maxflow.table}.}

\begin{table}[h!]
  \centering
\begin{tabular}{|r|l|l|l|l|l|l|l|l|l|l|l|l|}\hline
  $d$ & $2$ & $3$ & $4$ & $5$ & $6$ & $7$ & $8$ & $9$ & $10$ \\\hline
  lower bound & $0.7201$ & $0.7952$ & $0.8170$ & $0.8267$ & $0.8274$ & $0.8277$ & $0.8279$ & $0.8281$ & $0.8282$ \\\hline
  upper bound  & $0.8700$ & $0.8574$ & $0.8541$ & $0.8455$ & $0.8446$ & $0.8431$ & $0.8421$ & $0.8419$ & $0.8413$\\
  \hline
\end{tabular}
\caption {\label{tbl:maxflow.table}
  \bachir{Upper and lower bounds on the optimal value of the time-varying max-flow problem in \eqref{eq:tv.maxflow}.  In the first row, we report the objective value of the best polynomial solution of degree $d$. In the second row, we report the optimal value of the dual problem in \eqref{eqn:ud.tvsdp} at level $d$.}}
\end{table}

Note from the two tables that the objective value of the degree$\text{-}10$ polynomial solution we have found is guaranteed to be within \bachir{$2\%$} of the best objective value possible. The running time of our largest SDPs on a standard laptop with the solver MOSEK \cite{aps_mosek_2017} is in the order of a second. If we increase the degree much beyond $10$, our solver runs into numerical issues. This is not surprising as we are formulating our SDPs using the standard monomial basis. Much improvement is possible on the implementation front using e.g. the ideas in \cite{lofberg_coefficients_2004, papp2013semidefinite, papp_semi-infinite_2017, papp_sum--squares_2017}. Such implementation improvements are left for future work.

\subsection{A time-varying wireless coverage problem}
\label{sec:wireless.coverage}
In our second example, we present an application to wireless coverage of a targeted geographical region which moves over time. This is a time-varying generalization of problems considered in \cite{commander_optimization_2007, commander_wireless_2007, commander_jamming_2008, ahmadi_applications_2016}.
In this setting, we have $n_T$ wireless electromagnetic transmitters located at known locations $\bar T_i = (\bar x_i, \bar y_i)$ on the plane.
 Each transmitter $i \in \{1, \ldots, n_T\}$ is an omnidirectional power source providing a signal strength of \(E_i(t, x, y)\) at time \(t\) in location \((x, y)\) on the plane. 
Laws of electromagnetic wave propagation stipulate that
$$E_i(x, y, t)= \frac{c_i(t)}{(x - \bar x_i)^2 + (y - \bar y_i)^2},$$
where \(c_i(t)\), which is our decision variable, is the transmission power of the transmitter \(i\) at time \(t\).
There are $n_R$ regions on the plane that move over time and that need to be covered with sufficient signal strength. For \(j \in \{1, \ldots, n_R\}\) and $t \in [0, 1]$, we define each such region \(\mathcal B_j(t)\) with \(k_j\) polynomial inequalities:
$$\mathcal B_j(t) \coloneqq \{ (x, y) \in \mathbb R^2 |\quad  g_{t, j, k}(x, y) \ge 0, k = 1, \ldots, k_j\}.$$
Here, for $j=1,\ldots,n_R$, $k=1,\ldots,k_j$, $g_{t, j, k}(x, y)$ is a polynomial in $(x, y)$ whose coefficients depend on $t$. We further assume that for $j=1,\ldots,n_R$ and for all $t \in [0, 1]$,
$$g_{t, j, 1}(x, y) = r^2 - x^2 - y^2$$
for some large enough scalar $r$. 

Our goal is to ensure that for all time $t \in [0, 1]$, the strength of the signal in all regions $\mathcal B_j(t)$ is at least a given threshold $C$. In other words, our constraints in this problem are

\begin{equation}
\label{eq:Exyt}
\left.\begin{array}{rl}
E(x, y, t) \coloneqq \sum_{i=1}^{n_T} E_i(x, y, t) \ge C & \forall (x, y) \in \mathcal B_j(t), \forall j \in \{1, \ldots, n_R\}\\
c_i(t) \ge 0& \forall i \in \{1, \ldots, n_T\}
\end{array}\right\} \quad\forall t \in [0, 1] \text{ a.e.}.
\end{equation}
Our objective is to minimize the total cost of power generation, which is directly proportional to  $$\int_{0}^1 \sum_{i=1}^{n_T} c_i(t) \; {\rm d}t.$$

Notice that the first inequality in \eqref{eq:Exyt} is an inequality involving rational functions. Upon taking common denominators, we can reformulate this constraint as
\begin{equation}
  \label{eqn:poly.p}
  \begin{aligned}
    p_t(x, y) \coloneqq -C \prod_{i=1}^{n_T} [(x - \bar x_i)^2 + (y - \bar y_i)^2] + \sum_{i=1}^{n_T}c_i(t) \prod_{k \ne i} [(x - \bar x_k)^2 + (y - \bar y_k)^2]  \ge 0 \\ \forall (x, y)\in \mathcal B_j(t), \forall j =1,\ldots,n_R, \forall t \in [0, 1] \text{ a.e.}.
    \end{aligned}
\end{equation}
Note that $p_t(x, y)$ is a polynomial in $(x, y)$ whose coefficients depend on $t$.
Let $v_{\tilde d}$ denote the vector of monomials in $(x, y)$ of degree up to $\tilde d$, i.e.
$$v_{\tilde d} \coloneqq v_{\tilde d}(x, y) = (1, x, \ldots, x^{\tilde d}, xy, \ldots, x^{\tilde d-1}y, \ldots, y^{\tilde d})^T.$$
It is easy to check that for fixed $j \in \{1, \ldots, n_R\}, t \in [0, 1]$, existence of  positive semidefinite matrices $P^{(j)}_0(t), \ldots, P^{(j)}_{k_j}(t)$ satisfying the polynomial identity
\begin{equation}p_t(x, y) = v_{\tilde d}(x,y)^TP^{(j)}_{0}(t) v_{\tilde d}(x,y) + \sum_{k=1}^{k_j}  v_{\tilde d}(x,y)^TP^{(j)}_{k}(t) v_{\tilde d}(x,y) g_{t, j, k}(x,y)\label{eqn:certificate.ptj}\end{equation}
implies the constraint in \eqref{eqn:poly.p}.
Conversely, for every fixed $j \in \{1, \ldots, n_R\}$ and  $t \in [0, 1]$, Putinar's Positivstellensatz \cite{berr_positive_2001} implies that if the constraint in \eqref{eqn:poly.p} is satisfied strictly, one can always find a nonnegative integer $\tilde d$ and matrices $P^{(j)}_0(t), \ldots, P^{(j)}_{k_j}(t)$ that satisfy \eqref{eqn:certificate.ptj}.

For any fixed $\tilde d \in \mathbb N$, our overall problem is the following TV-SDP:
\begin{equation}
  \label{eqn:wireless}
  \begin{array}{ll}
    \underset{c_i,  P^{(j)}_{k}}{\min}
    & \int_{0}^1 \sum_{i=1}^{n_T}c_i(t) \; {\rm d}t \\
    & c_i \in \LintR{1} \quad i = 1,\ldots, n_T\\
    & P^{(j)}_{k} \in {\bf S}^{\frac{(\tilde d +1)(\tilde d +2)}2}  \quad k = 0,\ldots, k_j, \; j=1,\ldots,n_R\\\\
    &\left.\begin{array}{ll}
     c_i(t) \ge 0 & i=1,\ldots, n_T \\
    p_t(x, y) = v_{\tilde d}^TP^{(j)}_{0}(t)v_{\tilde d} + \sum_{k=1}^{k_j} g_{t, j, k}(x, y)v_{\tilde d}^TP^{(j)}_{k}(t)v_{\tilde d}  &\forall (x, y) \in \mathbb R^2, \; j=1,\ldots, n_R\\
    P^{(j)}_{k}(t) \succeq 0 & k=0,\ldots, k_j, j=1,\ldots, n_R
     \end{array}\right\}\; \forall t \in [0, 1] \text{ a.e.}.
  \end{array}
\end{equation}
Note that constraint \eqref{eqn:poly.p} that appears in the TV-SDP in \eqref{eqn:wireless} is an equality between two polynomials in $(x, y)$. Since two polynomials are equal if and only if their coefficients match, this constraint can be rewritten as a finite number of linear equations in our decision variables.

We now solve a numerical example with the following data:
$$C = 1,  n_T = 2, \bar T_1 = (0, 0), \bar T_2 = (5, 5), n_R = 2, k_1 = k_2 = 2,$$
$$r = 10, \; g_{t, 1, 2}(x, y) = 1 - \left( (x-3t+3)^2 + (y - 5t)^2\right), \; g_{t, 2, 2}(x, y) = 1 - \left(x^2 + (y - 5t+1))^2\right).$$
In other words, our two regions are \bachir{disks} of unit radius whose centers move with time.


We start by finding polynomial \bachir{solutions} $c_1, c_2 \in \mathbb R_d[t]$ that \bachir{satisfy} the nonnegativity and the signal strength requirements in \eqref{eq:Exyt}. For this, we solve the TV-SDP in \eqref{eqn:wireless} with $\tilde d = 1$. Using the methodology of Section \ref{sec:tvsdp_is_sdp}, we solve semidefinite programs (as given in Theorem \ref{thm:tvsdp_as_sdp}) to obtain the best polynomial solution of degree \(d \in \{2, 3,\ldots,10\}\). The objective values of the optimal solutions are reported in Table \ref{tbl:wireless}.

\begin{table}
  \centering
\begin{tabular}{|r|l|l|l|l|l|l|l|l|l|l|l|l|}\hline
  $d$ & $2$ & $3$ & $4$ & $5$ & $6$ & $7$ & $8$ & $9$ & $10$ \\\hline
  & $+\infty$ &$56.64$ &$54.52$ &$54.43$ &$54.14$ &$54.14$ &$53.95$ &$53.94$ &$53.93$\\\hline
\end{tabular}
\caption {\label{tbl:wireless} Objective values of optimal polynomial solutions of degree $d$ to the time-varying wireless coverage problem in \eqref{eqn:wireless}.}
\end{table}

Note that if we do not allow the solution to depend on time (or even if we allow it to depend on time as a polynomial of degree less than $3$), then the TV-SDP in \eqref{eqn:wireless} becomes infeasible. As we increase the degree, the problem becomes feasible and the objective value improves.



Figure \ref{img:wireless} demonstrates a sanity check on our solution at six snapshots of time. Indeed, the two regions $\mathcal B_1(t)$ and $\mathcal B_2(t)$ are receiving a signal of strength of at least $1$.

\newcommand\wirelesspng[1]{\includegraphics[trim={.75cm 0cm 1cm 0cm}, clip, width=.3\textwidth]{includes/wireless/wireless#1.png}}
\begin{figure}[htp]
\centering
\wirelesspng{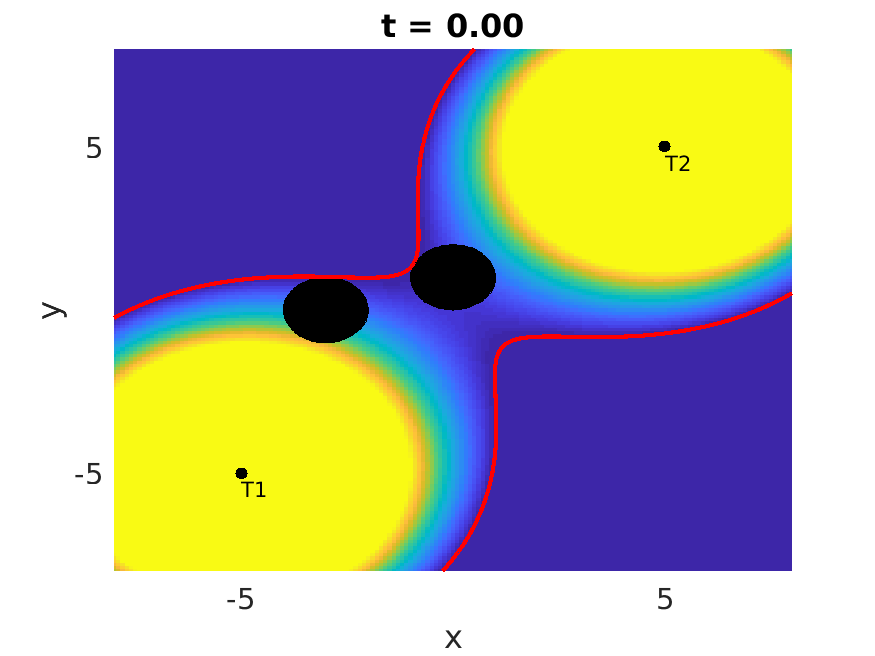}\wirelesspng{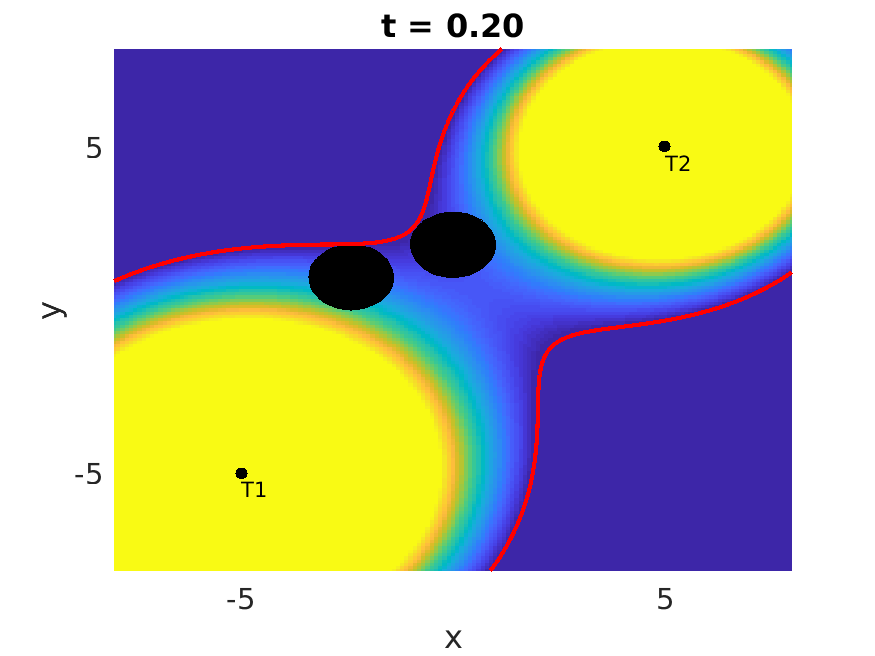}\wirelesspng{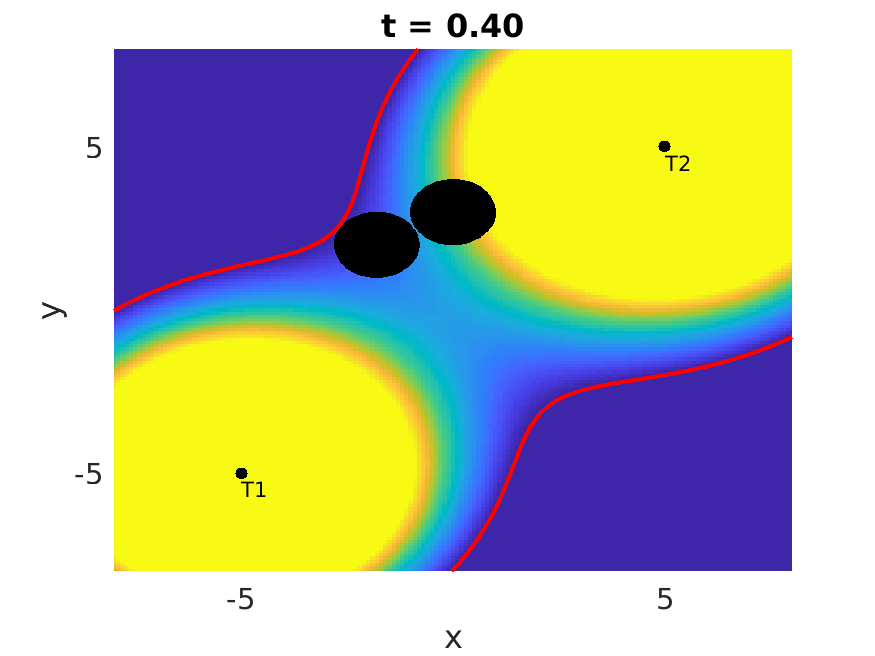}\\
\wirelesspng{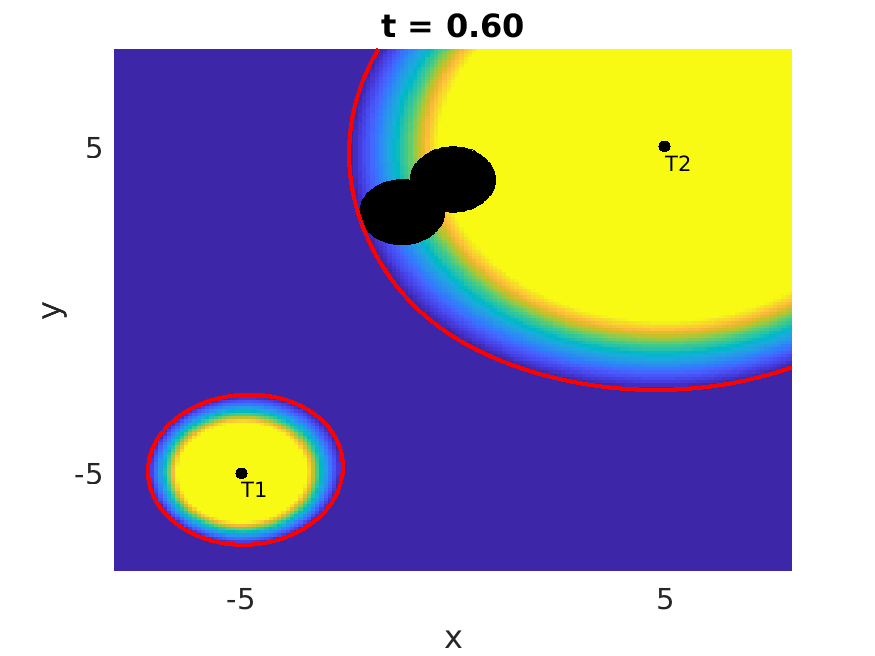}\wirelesspng{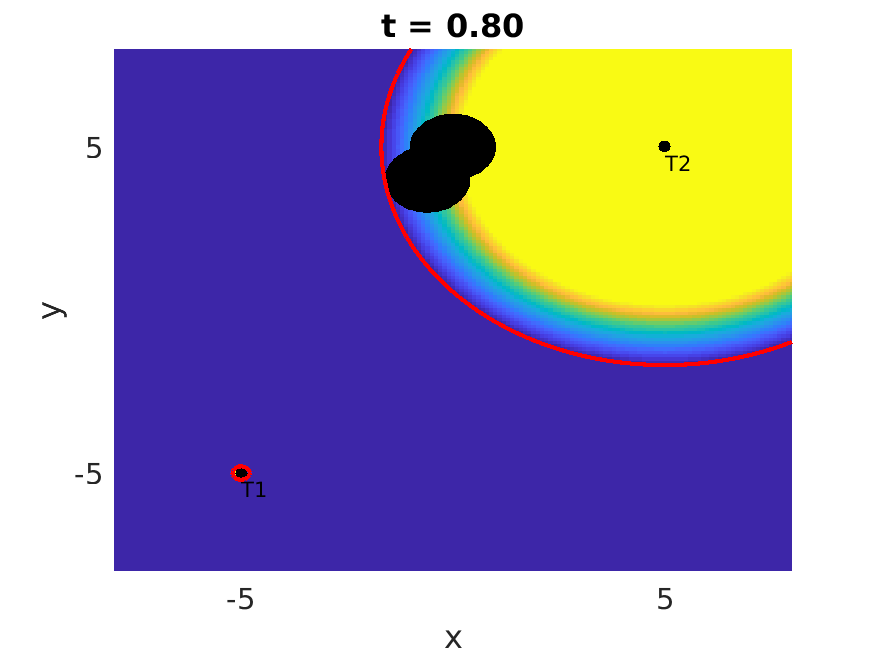}\wirelesspng{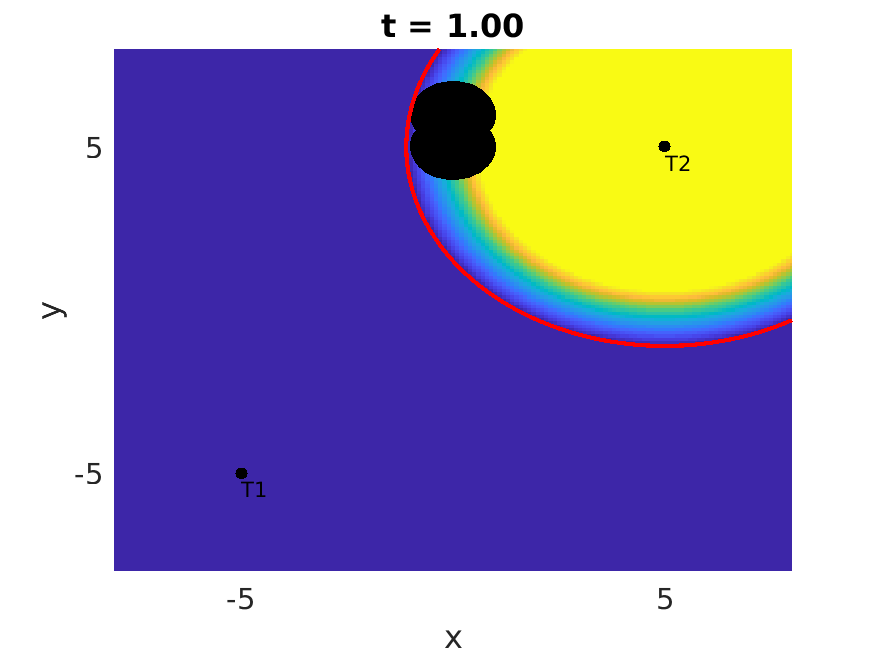}
\caption{\label{img:wireless} Six time snapshots---at $t=0, \frac15, \frac25, \frac35, \frac45, 1$---of the wireless coverage obtained by the best polynomial solution of degree $10$. The two time-varying regions $\mathcal B_1(t)$ and $\mathcal B_2(t)$ that need to  receive a signal strength of at least $1$ at all times $t \in [0, 1]$ are colored in black. The heatmap in the background demonstrates the signal strength at each location with light yellow representing high and dark blue representing low signal strengths. The region delimited by the red curves is guaranteed to receive a signal strength of at least $1$.}
\end{figure}

To have an idea of how far our best polynomial solution of degree $10$ is from being optimal to the TV-SDP in \eqref{eqn:wireless}, we solve the dual problem \eqref{eqn:ud.tvsdp} presented in Section \ref{sec:dual-approach}. After some rewriting, this dual problem at level $d$ becomes the following SDP:
\begin{equation}
  \label{eqn:dual.wireless}
  \begin{array}{ll}
    &\underset{c_i,  P^{(j)}_{k} \text{of degree $d+2$}}{\min}
     \int_{0}^1 \sum_{i=1}^{n_T}c_i(t) \; {\rm d}t \\
    &\text{subject to}\\
    \small
    &\quad\begin{array}{lr}
     \adjalpha (c_i) \succeq 0, \adjbeta (c_i) \succeq 0 & i=1,\ldots, n_T \\
      \adjalpha\left(p_t(x, y) - v_{\tilde d}^TP^{(j)}_{0}(t)v_{\tilde d} + \sum_{k=1}^{k_j} g_{t, j, k}(x, y)v_{\tilde d}^TP^{(j)}_{k}(t)v_{\tilde d}\right) = 0  &\forall (x, y) \in \mathbb R^2, \; j=1,\ldots, n_R\\
      \adjbeta\left(p_t(x, y) - v_{\tilde d}^TP^{(j)}_{0}(t)v_{\tilde d} + \sum_{k=1}^{k_j} g_{t, j, k}(x, y)v_{\tilde d}^TP^{(j)}_{k}(t)v_{\tilde d}\right) = 0 &\forall (x, y) \in \mathbb R^2,\; j=1,\ldots, n_R\\
    \adjoint{\alpha_d^m}(P^{(j)}_{k})\succeq 0, \adjoint{\beta_d^m}(P^{(j)}_{k})\succeq 0 & k=1,\ldots, n_T, \; j=1,\ldots, n_R,
     \end{array}
  \end{array}
\end{equation}
\normalfont
where $m = \frac{(\tilde d +1)(\tilde d +2)}2$. Note that the second and third set of constraints are requiring a polynomial matrix in $(x, y)$ whose coefficients depend linearly on the decision variables to be identically zero. Once again, this is simply a finite numbers of equality constraints.

The optimal value of problem
  \eqref{eqn:dual.wireless} with $d = 10$ is equal to
  $52.66$. This tells us that the objective value of the
  degree-10 polynomial solution \bachir{reported in
  Table~\ref{tbl:wireless}} is within $2.5\%$ of the optimal
  value of the TV-SDP in \eqref{eqn:wireless}.


\subsection{Bi-objective SDP and Pareto curve approximation}
In our third and last example, we formulate a bi-objective (non time-varying) semidefinite program as a time-varying SDP. 
 
A bi-objective semidefinite program is a standard SDP that involves two objective functions. More precisely, we are concerned with the simultaneous maximization of two objective functions
$$\langle c_1, x \rangle \text{ and } \langle c_2, x \rangle$$
over the feasible set
$$\mathcal F \coloneqq \{x \in \mathbb R^n \; | \; Fx \coloneqq A_0 + \sum_{i=1}^n x_i A_i \succeq 0\},$$
where $A_0, \ldots, A_n$ are given by $m \times m$ symmetric matrices. 
In general there exists no single solution $x$ that maximizes both objective functions at the same time. As a trade-off, one is interested in solving the following problem 
\begin{equation}
  \label{eq:weighted.bi.objective}
y(t) \coloneqq \quad \begin{array}{ll@{}ll}
  \underset{x \in \mathbb R^n}{\max} & \langle c_1, x \rangle & \\
  \text{subject to}&  \langle c_2, x \rangle \ge  t \text{ and }
  &  Fx  \succeq 0,\\
\end{array}
\end{equation}
for various values of $t$. 
 In the case where $\mathcal F$ is compact, we can  without loss of generality take $t$ to vary in $[0, 1]$ after a possible rescaling. This gives rise to the following trade-off curve, which we refer to as the Pareto curve:
$$PC \coloneqq \{( t, y(t) )\; | \; t \in [0, 1]\}.$$
Any point on this curve tells us that in order to improve the first objective function beyond $y(t)$, the second objective needs to necessarily be smaller than $t$. We are interested in a one-shot approximation of the entire Pareto curve as oppposed to sampling points on it and solving several independent SDPs.
Such an approach has been taken before for multi-objective LPs in  \cite{gorissen_approximating_2012}, and for bi-objective polynomial optimization problems in \cite{magron_approximating_2014}.

%

To get the Pareto curve in one shot, we can solve the following TV-SDP
\begin{equation}
  \label{eq:tvsdp.bi.objective}
\begin{array}{llll}
  \underset{x \in \LintRn}{\max}
  & \int_0^1 \langle c_1, x(t) \rangle {\rm d}t  & \\
  \text{subject to}
  &
    \left.\begin{array}{l}
      \langle c_2, x(t) \rangle \ge t\\
      Fx(t)  \succeq 0\end{array}\right\} \quad
  \forall t \in [0, 1] \text{ a.e.}.
\end{array}
\end{equation}
If $x \in \LintRn$ is any feasible solution to this TV-SDP, then
$$\langle c_1, x(t) \rangle \le y(t) \; \forall t \in [0, 1] \text{ a.e.}.$$
In other words, any feasible solution to the TV-SDP in \eqref{eq:tvsdp.bi.objective} gives a lower to the Pareto curve almost every where on $[0, 1]$.
Furthermore, if $x^\opt$ is an optimal solution to the same TV-SDP 
(whose existence is guaranteed by Theorem \ref{thm:optim-value-attained} when $\mathcal F$ is compact), then
$$\langle c_1, x^\opt(t) \rangle = y(t) \; \forall t \in [0, 1] \text{ a.e.}.$$
Let $x^d \in \mathbb R^n_d[t]$ be an optimal solution to \eqref{eq:tvsdp.bi.objective} when the search space is restricted to polynomials of degree at most $d$.
We know from Theorem \ref{thm:poly.optimal} that, under the strict feasibility assumption\footnote{In this setup, this assumption is equivalent to existence of positive scalar $\varepsilon$  and a vector $x^s \in \mathbb R^n$ such that $F x^s \succeq \varepsilon I$ and $\langle c_2, x^s \rangle \ge 1+\varepsilon$.} in Definition \ref{def:strict_feasibility_sdp}, 
$$\int_0^1 y(t) - \langle c_1, x^d(t) \rangle \; {\rm d}t \rightarrow 0 \text{ as } d \rightarrow \infty.$$
Moreover, the optimal value of the dual problem of the TV-SDP in \eqref{eq:tvsdp.bi.objective} at level $d$, as described in Section \ref{sec:dual-approach}, gives an upper bound on the area under the Pareto curve. Under the assumption that the set $\mathcal F$ is bounded in the infinity norm by $\gamma$, then once the constraint
$\|x\|_\infty \le \gamma$ 
is added to the TV-SDP in \eqref{eq:tvsdp.bi.objective}, the optimal values of the associated dual problems converge to the area under the Pareto curve as $d \rightarrow \infty$ (see Theorem \ref{thm:strong.duality}).



As a concrete example of a bi-objective SDP, we consider the Markowitz portfolio selection problem \cite{markowitz_portfolio_1952}. We model $n$ tradable assets as  a nondegenerate $n\text{-variate}$ Gaussian random variable with average return $r \in \mathbb R^n$ and (positive definite) covariance matrix $\Sigma \in \mathcal S^{n}$. Given the data $r$ and $\Sigma$ as input, the goal is to choose a portfolio (i.e. an allocation of $x_i$ fraction of our total funds to asset $i \in \{1, \ldots, n\}$) that maximizes the average return $r^Tx$ while simultaneously minimizing the variance $x^T\Sigma x$.

We can formulate this problem as a bi-objective optimization problem, with variables
$$\begin{pmatrix}u,&x_1,&\ldots,&x_n\end{pmatrix}^T \in {\mathbb R^{n+1}},$$
constraints
$$x \ge 0, \sum_{i=1}^nx_i \le 1, \; x^T \Sigma x \le u,$$
and two objective functions
$$r^Tx \text{ and } -u.$$

The Pareto curve is therefore given by $\{ (t, y(t)) \; | \; t \in [0, 1]\}$, where 
\begin{equation}
  \label{eq:markowitz.pc}
 \begin{array}{llll}
  y(t) \coloneqq &\underset{x \in \mathbb R^n, u \in \mathbb R}{\max} & r^Tx & \\
  &\text{subject to}&x \ge 0\\
            &&\sum_{i=1}^nx_i \le 1\\
            &&u \le t\\
            &&\begin{pmatrix}u& x^T\\x&\Sigma^{-1}\end{pmatrix} \succeq 0.
\end{array}
\end{equation}

The TV-SDP in \eqref{eq:tvsdp.bi.objective} that gives this Pareto curve in one shot can therefore be written as
\begin{equation}
  \label{eqn:markowitz}
\begin{array}{llll}
  \underset{x \in \LintRn, u \in \LintR{1}}{\max}
  & \int_0^1 r^Tx(t) {\rm d}t  & \\
  \text{subject to}\\
  &\left.\begin{array}{l}
          x(t) \ge 0\\
          \sum_{i=1}^nx_i(t) \le 1\\
           u(t) \le t\\
           \begin{pmatrix}u(t)& x(t)^T\\x(t)&\Sigma^{-1}\end{pmatrix} \succeq 0\\
        \end{array}\quad \right\} \quad
  \forall t \in [0, 1] \text{ a.e.}.
\end{array}
\end{equation}

We numerically solve an example with $n=5$ assets,
\[
   r = \small\begin{pmatrix}0.4170,&0.7203,&0.0001,&0.3023,&0.1468\end{pmatrix}^T\normalfont,
  \Sigma =   \small\begin{pmatrix}
    6.0127&-0.7381&-0.5441&-4.9189& 1.7855\\
    -0.7381& 9.8904&-0.7946& 0.2481&-5.5214\\
    -0.5441&-0.7946& 5.1961&-3.6240& 1.5820\\
    -4.9189& 0.2481&-3.6240&10.4637& 1.7840\\
    1.7855&-5.5214& 1.5820& 1.7840&15.8475\\
  \end{pmatrix}.
  \normalfont
\]

The entries of the vector $r$ were generated independently from the uniform distribution over $[0, 1]$. The matrix $\Sigma$ was obtained by first generating a $5 \times 5$ matrix $V$ whose entries were sampled independently from the uniform distribution over $[0, 3]$, and then letting $\Sigma = VV^T$.

Using Theorem \ref{thm:tvsdp_as_sdp}, we solve a semidefinite program that finds the best
the best polynomial solution   of degree than $10$ to the TV-SDP in \eqref{eq:tvsdp.bi.objective}. 
 The objective value that we achieve is $0.3210$, and the resulting optimal solution $x^{\text{poly}, 10} \in \mathbb R^5_{10}[t]$ is plotted in Figure \ref{fig:markowitz.sol}.
 In Figure \ref{fig:pareto.sol}, we plot $r^Tx^{\text{poly}, 10}(t)$, which is a point-wise lower approximation to the true Pareto curve. We also find eleven equally-spaced points on the exact Pareto curve, by solving the problem in \eqref{eq:markowitz.pc} at $t \in \{0, 0.1, \ldots, 1\}$. Notice that our approximation to the Pareto curve obtained from the best polynomial solution of degree $10$ is almost perfect at these eleven sample points. 

 To get a formal upper bound on the area enclosed between $ \{(t, r^Tx^{\text{poly}, 10}(t)) \; | \; t \in [0, 1]\}$ and the true Pareto curve $\{(t, y(t)) \; | \; t \in [0, 1]\}$, we solve the dual problem \eqref{eqn:ud.tvsdp} presented in Section \ref{sec:dual-approach}. After some rewriting, this dual problem at level $d$ is equivalent to the following SDP (cf. Theorem \ref{thm:dual.is.sdp}):

\begin{equation}
  \label{eqn:markowitz.dual}
\begin{array}{llll}
  \underset{x \in \mathbb R_d^n[t], u \in \mathbb R_d[t]}{\max}
  & \int_0^1 r^Tx(t) {\rm d}t  & \\
  \text{subject to}
  &\begin{array}{ll}
          \adjoint{\alpha_d^1} (x_i) \succeq 0,& \adjoint{\beta_d^1} (x_i) \succeq 0 \quad i=1,\ldots,n\\
           \adjoint{\alpha_d^1}(1-\sum_{i=1}^nx_i) \succeq 0,&\adjoint{\beta_d^1}  (1-\sum_{i=1}^nx_i) \succeq 0\\           
          \adjoint{\alpha_d^1}(t-u(t)) \succeq 0,& \adjoint{\beta_d^1} (t-u(t)) \succeq 0\\
           \adjoint{\alpha_d^{n+1}} \begin{pmatrix}u(t)& x(t)^T\\x(t)&\Sigma^{-1}\end{pmatrix} \succeq 0, &\adjoint{\beta_d^{n+1}}  \begin{pmatrix}u(t)& x(t)^T\\x(t)&\Sigma^{-1}\end{pmatrix} \succeq 0.
        \end{array}
\end{array}
\end{equation}

The optimal value of problem \eqref{eqn:markowitz.dual} with $d = 10$ is equal to $0.3232$, which tells us that
$$\int_0^1 \left(y(t)  - r^Tx^{\text{poly}, 10}(t)\right)  \; {\rm d}t \le \frac{1}{100} \int_0^1 y(t) \; {\rm d}t.$$

\begin{figure}[H]
    \centering
    \begin{subfigure}[b]{0.4\textwidth}
        \includegraphics[width=\textwidth]{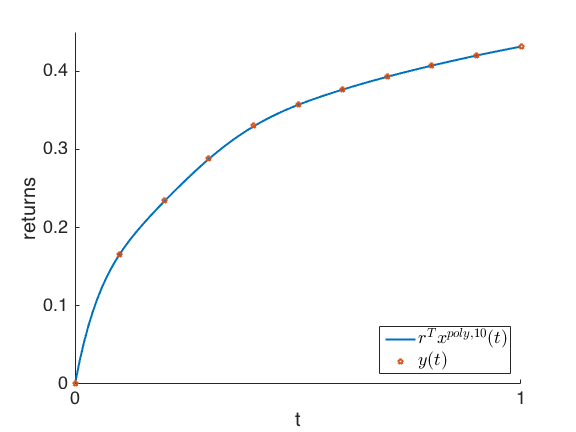}
        \caption{\label{fig:pareto.sol}The return $r^Tx^{\text{poly},10}(t)$ obtained by best polynomial solution of degree $\le 10$.\\}
    \end{subfigure}
    ~
    \begin{subfigure}[b]{0.4\textwidth}
        \includegraphics[width=\textwidth]{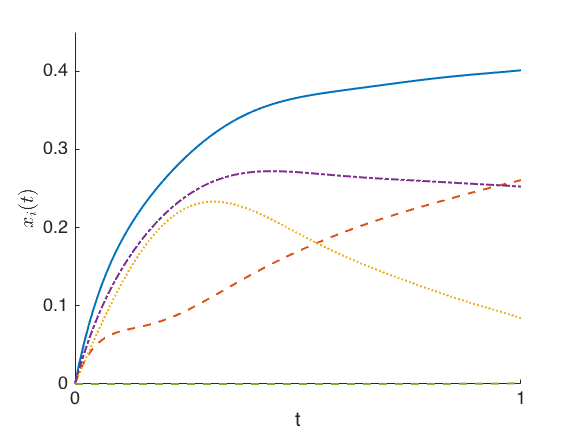}
        \caption{\label{fig:markowitz.sol} The allocations $x_i^{\text{poly},10}(t)$ for the different assets obtained by the best polynomial solution of degree $\le 10$.}
    \end{subfigure}
    \caption{The optimal polynomial solution of degree less than 10 and its associated approximation to the Pareto curve for the Markowitz portfolio selection problem.}\label{fig:pareto}
\end{figure}



\section{Future Research Directions}
\label{sec:conclusion}
We end by mentioning a few questions that are left for future research.
\bachir{We believe there is much research to be done to extend some of the fundamental structural results from the continuous linear programming literature (e.g., results related to duality theory or the structure of optimal solutions) to the case of TV-SDPs. As a concrete example, we would be interested in knowing to what extent the duality theory of Pullan \cite{pullan1996duality} can carry over to the TV-SDP setting.} 

\bachir{Closer to the focus of this paper,} we have shown in Theorem \ref{thm:poly.optimal} that under the strict feasibility assumption in Definition \ref{def:strict_feasibility_sdp}, the sequence of objective values of the best polynomial solution of degree $d$ converges to the optimal value of the TV-SDP as $d \rightarrow \infty$. If we are interested in a feasible solution with (additive or multiplicative) error bounded by $\alpha$, how large should we take $d$ to be as a function of $\alpha$ and other problem parameters? The answer to this question would likely have a dependence on the scalar $\varepsilon$ in Definition \ref{def:strict_feasibility_sdp}. Is there an efficient method for obtaining a lower bound on $\varepsilon$, or even checking the strict feasibility assumption?
Lastly, we are interested in knowing whether the strict feasibility assumption in Theorem \ref{thm:poly.optimal} can be weakened, for instance, to existence of a feasible polynomial solution.

Similarly in Theorem \ref{thm:strong.duality}, we have shown that under a boundedness assumption, the sequence of optimal values of our dual problem at level $d$ converges from above to the optimal value of the TV-SDP. It would be interesting to study the convergence rate of this sequence. We also would like to know if the boundedness assumption is needed for convergence, and whether the bound constraints need to be explicitly added to the TV-SDP as we do now.

Finally, at a more basic level, what is the complexity (in the Turing model of computation) of testing feasibility of a continuous linear program with polynomially-varying data? Here, the maximum degree of the polynomials in the data can either be fixed or part of the input. The reason we do not ask this complexity question for TV-SDPs is that the question is well known to be open even for standard SDPs (see e.g. \cite{klerk_aspects_2002}).

\section*{Acknowledgement}
The authors are grateful to Farid Alizadeh, Daniel
Bienstock, Robert J. Vanderbei, and Ramon van Handel for
insightful questions and comments.  The authors
  are also thankful to two anonymous referees for their
  careful reading of the paper and many valuable suggestions
  and comments.

\bibliographystyle{custom_plain}
\bibliography{tvsdp}

\begin{thebibliography}{10}

\bibitem{adams_sobolev_2003}
R.~A. Adams and J.~J.~F. Fournier.
\newblock {\em Sobolev {Spaces}}.
\newblock Elsevier, 2003.

\bibitem{ahmadi_applications_2016}
A.~A. Ahmadi and A.~Majumdar.
\newblock Some applications of polynomial optimization in operations research
  and real-time decision making.
\newblock {\em Optimization Letters}, 10(4):709--729, 2016.

\bibitem{ahmadi_complete_2013}
A.~A. Ahmadi and P.~Parrilo.
\newblock A {complete} characterization of the gap between convexity and
  {sos}-convexity.
\newblock {\em SIAM Journal on Optimization}, 23(2):811--833, January 2013.

\bibitem{anderson1987linear}
E.~J. Anderson and P.~Nash.
\newblock {\em Linear programming in infinite-dimensional spaces: theory and
  applications}.
\newblock John Wiley \& Sons, 1987.

\bibitem{anderson_continuous-time_1989}
E.~J. Anderson and A.~B. Philpott.
\newblock A continuous-time network simplex algorithm.
\newblock {\em Networks}, 19(4):395--425, 1989.

\bibitem{anstreicher1984generation}
K.~M. Anstreicher.
\newblock Generation of feasible descent directions in continuous time linear
  programming.
\newblock {\em Tech. Report, SOL 83-18, Department of Operations Research,
  Stanford University, Stanford, CA}, 1984.

\bibitem{aylward2007explicit}
E.~M. Aylward, S.~M. Itani, and P.~A. Parrilo.
\newblock Explicit {SOS} decompositions of univariate polynomial matrices and
  the {Kalman-Yakubovich-Popov} lemma.
\newblock In {\em Proceedings of the 46th IEEE Conference on Decision and
  Control}, pages 5660--5665, 2007.

\bibitem{bampou2011scenario}
D.~Bampou and D.~Kuhn.
\newblock Scenario-free stochastic programming with polynomial decision rules.
\newblock In {\em Proceedings of the IEEE Conference on Decision and Control},
  pages 7806--7812, 2011.

\bibitem{bampou_polynomial_2012}
D.~Bampou and D.~Kuhn.
\newblock Polynomial approximations for continuous linear programs.
\newblock {\em SIAM Journal on Optimization}, 22(2):628--648, January 2012.

\bibitem{bellman_bottleneck_1953}
R.~Bellman.
\newblock Bottleneck problems and dynamic programming.
\newblock {\em Proceedings of the National Academy of Sciences of the United
  States of America}, 39(9):947--951, September 1953.

\bibitem{berr_positive_2001}
R.~Berr and T.~W{\"o}rmann.
\newblock Positive polynomials on compact sets.
\newblock {\em Manuscripta Mathematica}, 104(2):135--143, February 2001.

\bibitem{bertsimas2011hierarchy}
D.~Bertsimas, D.~A. Iancu, and P.~A. Parrilo.
\newblock A hierarchy of near-optimal policies for multistage adaptive
  optimization.
\newblock {\em IEEE Transactions on Automatic Control}, 56(12):2809--2824,
  2011.

\bibitem{buie_numerical_1973}
R.~N. Buie and J.~Abrham.
\newblock Numerical solutions to continuous linear programming problems.
\newblock {\em Zeitschrift f{\"u}r Operations Research}, 17(3):107--117, July
  1973.

\bibitem{commander_optimization_2007}
C.~W. Commander.
\newblock {\em Optimization problems in telecommunications with military
  applications.}
\newblock {PhD} {Thesis}, University of Florida, 2007.

\bibitem{commander_jamming_2008}
C.~W. Commander, P.~M. Pardalos, V.~Ryabchenko, O.~Shylo, S.~Uryasev, and
  G.~Zrazhevsky.
\newblock Jamming communication networks under complete uncertainty.
\newblock {\em Optimization Letters}, 2(1):53--70, January 2008.

\bibitem{commander_wireless_2007}
C.~W. Commander, P.~M. Pardalos, V.~Ryabchenko, S.~Uryasev, and G.~Zrazhevsky.
\newblock The wireless network jamming problem.
\newblock {\em Journal of Combinatorial Optimization}, 14(4):481--498, November
  2007.

\bibitem{klerk_aspects_2002}
E.~de~Klerk.
\newblock {\em Aspects of {Semidefinite} {Programming}: {Interior} {Point}
  {Algorithms} and {Selected} {Applications}}.
\newblock Springer, 2002.

\bibitem{dette_matrix_2002}
H.~Dette and W.~J. Studden.
\newblock Matrix measures, moment spaces and {Favard}'s theorem for the
  interval [0,1] and [0, $\infty$).
\newblock {\em Linear Algebra and its Applications}, 345(1-3):169--193, April
  2002.

\bibitem{fleischer_efficient_2005}
L.~Fleischer and J.~Sethuraman.
\newblock Efficient algorithms for separated continuous linear programs: the
  multicommodity flow problem with holding costs and extensions.
\newblock {\em Mathematics of Operations Research}, 30(4):916--938, November
  2005.

\bibitem{gorissen_approximating_2012}
B.~L. Gorissen and D.~den Hertog.
\newblock Approximating the {Pareto} set of multiobjective linear programs via
  robust optimization.
\newblock {\em Operations Research Letters}, 40(5):319--324, September 2012.

\bibitem{grinold_continuous_1969}
R.~C. Grinold.
\newblock Continuous programming part one: {linear} objectives.
\newblock {\em Journal of Mathematical Analysis and Applications},
  28(1):32--51, October 1969.

\bibitem{hilbert_beitrag_1894}
D.~Hilbert.
\newblock Ein {B}eitrag zur {T}heorie des {Legendre}'schen {P}olynoms.
\newblock {\em Acta Mathematica}, 18:155--159, 1894.

\bibitem{kojima_sums_2003}
M.~Kojima.
\newblock Sums of squares relaxations of polynomial semidefinite programs.
\newblock {\em Research report B-397, Dept. of Mathematical and Computing
  Sciences, Tokyo Institute of Technology}, 2003.

\bibitem{lasserre_global_2001}
J.~B. Lasserre.
\newblock Global optimization with polynomials and the problem of moments.
\newblock {\em SIAM Journal on Optimization}, 11(3):796--817, January 2001.

\bibitem{lasserre_joint+marginal_2009}
J.~B. Lasserre.
\newblock A {\textquotedblleft}joint+marginal{\textquotedblright} approach to
  parametric polynomial optimization.
\newblock {\em SIAM Journal on Optimization}, 20, 2009.

\bibitem{lasserre_moments_2010}
J.~B. Lasserre.
\newblock {\em Moments, {Positive} {Polynomials} and {Their} {Applications}}.
\newblock World Scientific, 2010.

\bibitem{lehman_continuous_1954}
R.~S. Lehman.
\newblock On the continuous simplex method.
\newblock Technical Report RM-1386, Rand Corporations, Santa Monica., December
  1954.

\bibitem{levinson_class_1966}
N.~Levinson.
\newblock A class of continuous linear programming problems.
\newblock {\em Journal of Mathematical Analysis and Applications},
  16(1):73--83, October 1966.

\bibitem{lofberg_yalmip_2004}
J.~L{\"o}fberg.
\newblock {YALMIP} : a toolbox for modeling and optimization in {MATLAB}.
\newblock In {\em Proceedings of the {IEEE} {International} {Conference} on
  {Robotics} and {Automation}}, pages 284--289, September 2004.

\bibitem{lofberg_coefficients_2004}
J.~L{\"o}fberg and P.~A. Parrilo.
\newblock From coefficients to samples: a new approach to {SOS} optimization.
\newblock In {\em Proceedings of the {IEEE} {Conference} on {Decision} and
  {Control}}, volume~3, pages 3154--3159, December 2004.

\bibitem{luo_new_1998}
X.~Luo and D.~Bertsimas.
\newblock A new algorithm for state-constrained separated continuous linear
  programs.
\newblock {\em SIAM Journal on Control and Optimization}, 37(1):177--210,
  January 1998.

\bibitem{magron_approximating_2014}
V.~Magron, D.~Henrion, and J.~B. Lasserre.
\newblock Approximating {Pareto} curves using semidefinite relaxations.
\newblock {\em Operations Research Letters}, 42(6):432--437, September 2014.

\bibitem{markowitz_portfolio_1952}
H.~Markowitz.
\newblock Portfolio selection.
\newblock {\em The Journal of Finance}, 7(1):77--91, 1952.

\bibitem{aps_mosek_2017}
{MOSEK ApS}.
\newblock {\em The MOSEK optimization toolbox for {MATLAB} manual. Version
  8.1.}, 2017.

\bibitem{nesterov2000squared}
Y.~Nesterov.
\newblock Squared functional systems and optimization problems.
\newblock In {\em High performance optimization}, pages 405--440. Springer,
  2000.

\bibitem{papachristodoulou_sostools:_2013}
A.~Papachristodoulou, P.~A. Parrilo, P.~Seiler, J.~Anderson, G.~Valmorbida,
  S.~Prajna, and P.~Seiler.
\newblock {\em {SOSTOOLS}: {Sum} of {Squares} {Optimization} {Toolbox} for
  {MATLAB}}.
\newblock 2013.

\bibitem{papp_semi-infinite_2017}
D.~Papp.
\newblock Semi-infinite programming using high-degree polynomial interpolants
  and semidefinite programming.
\newblock {\em SIAM Journal on Optimization}, 27(3):1858--1879, January 2017.

\bibitem{papp2013semidefinite}
D.~Papp and F.~Alizadeh.
\newblock Semidefinite characterization of sum-of-squares cones in algebras.
\newblock {\em SIAM Journal on Optimization}, 23(3):1398--1423, 2013.

\bibitem{papp_sum--squares_2017}
D.~Papp and S.~Yildiz.
\newblock Sum-of-squares optimization without semidefinite programming.
\newblock {\em SIAM Journal on Optimization}, 29(1):822--851, 2019.

\bibitem{parrilo_semidefinite_2003}
P.~A. Parrilo.
\newblock Semidefinite programming relaxations for semialgebraic problems.
\newblock {\em Mathematical Programming}, 96(2):293--320, 2003.

\bibitem{perold_fundamentals_1978}
A.~F. Perold.
\newblock Fundamentals of a continuous time simplex method.
\newblock Technical Report SOL-78-26, Rand Corporations, Santa Monica., 1978.

\bibitem{pullan1993algorithm}
M.~C. Pullan.
\newblock An algorithm for a class of continuous linear programs.
\newblock {\em SIAM Journal on Control and Optimization}, 31(6):1558--1577,
  1993.

\bibitem{pullan1995forms}
M.~C. Pullan.
\newblock Forms of optimal solutions for separated continuous linear programs.
\newblock {\em SIAM Journal on Control and Optimization}, 33(6):1952--1977,
  1995.

\bibitem{pullan1996duality}
M.~C. Pullan.
\newblock A duality theory for separated continuous linear programs.
\newblock {\em SIAM Journal on Control and Optimization}, 34(3):931--965, 1996.

\bibitem{pullan2000convergence}
M.~C. Pullan.
\newblock Convergence of a general class of algorithms for separated continuous
  linear programs.
\newblock {\em SIAM Journal on Optimization}, 10(3):722--731, 2000.

\bibitem{shapiro_duality_2001}
A.~Shapiro.
\newblock On duality theory of conic linear problems.
\newblock In {\em Semi-{Infinite} {Programming}: {Recent} {Advances}}, pages
  135--165. Springer US, Boston, MA, 2001.

\bibitem{timan_theory_2014}
A.~F. Timan.
\newblock {\em Theory of {Approximation} of {Functions} of a {Real}
  {Variable}}.
\newblock Elsevier, 2014.

\bibitem{tyndall_duality_1965}
W.~F. Tyndall.
\newblock A duality theorem for a class of continuous linear programming
  problems.
\newblock {\em Journal of the Society for Industrial and Applied Mathematics},
  13(3):644--666, 1965.

\bibitem{tyndall_extended_1967}
W.~F. Tyndall.
\newblock An extended duality theorem for continuous linear programming
  problems.
\newblock {\em SIAM Journal on Applied Mathematics}, 15(5):1294--1298,
  September 1967.

\bibitem{walkden_lecture_nodate}
C.~Walkden.
\newblock Lecture notes on {Ergodic} {theory}.
\newblock {\em The University of Manchester}, 2018.

\bibitem{wang_separated_2009}
X.~Wang, S.~Zhang, and D.~Yao.
\newblock Separated {continuous} {conic} {programming}: {strong} {duality} and
  an {approximation} {algorithm}.
\newblock {\em SIAM Journal on Control and Optimization}, 48(4):2118--2138,
  January 2009.

\bibitem{weiss_simplex_2008}
G.~Weiss.
\newblock A simplex based algorithm to solve separated continuous linear
  programs.
\newblock {\em Mathematical Programming}, 115(1):151--198, September 2008.

\end{thebibliography}

\newpage
\end{document}